\documentclass[10pt]{amsart}

\usepackage{amsmath}
\usepackage{array,multirow,graphicx}
\usepackage{amsfonts}
\usepackage{amssymb}
\usepackage{amsthm}
\usepackage{url}
\usepackage[all]{xy}
\usepackage{dsfont} 
\usepackage{nicematrix}
\usepackage{enumitem}
\usepackage{graphicx}
\usepackage{caption}
\usepackage{subcaption}
\usepackage{comment} 
\usepackage{stmaryrd}
\usepackage{hyperref}
\usepackage{todonotes}
\usepackage{color}
\usepackage{tikz-cd}
\usepackage[backend=bibtex,style=alphabetic,url=false,doi=false,isbn=false,sorting=nyt,maxbibnames=9,giveninits=true]{biblatex}
\usepackage{MnSymbol}
\usepackage{mathtools}
\usepackage{float}
\allowdisplaybreaks

 \AtEveryBibitem{\clearfield{month}}

\AtEveryBibitem{%
  \clearlist{language}%
}

\numberwithin{equation}{section}

\newtheorem{proposition}{Proposition}[section]
\newtheorem{lemma}[proposition]{Lemma}
\newtheorem{theorem}[proposition]{Theorem}
\newtheorem{corollary}[proposition]{Corollary}

\theoremstyle{definition}
\newtheorem{remark}[proposition]{Remark}
\newtheorem{definition}[proposition]{Definition}

\DeclareMathOperator{\Aut}{Aut}

\DeclareMathOperator{\diag}{Diag} 

\newcommand{\R}{\mathbb{R}}

\newcommand{\Z}{\mathbb{Z}}

\newcommand{\Q}{\mathbb{Q}}

\newcommand{\pr}{\mathbb{P}}
\renewcommand{\epsilon}{\varepsilon}

\renewcommand{\phi}{\varphi}

\title{GIT stability of divisors in products of projective spaces}
\author[I. Karagiorgis]{Ioannis Karagiorgis}
\address{School of Mathematics and Statistics, University of Glasgow, University Place, Glasgow, G12 8QQ, United Kingdom}
\email{Ioannis.Karagiorgis@glasgow.ac.uk}

\author[T. A. Ortscheidt]{Theresa A. Ortscheidt}
\address{School of Mathematics and Statistics, University of Glasgow, University Place, Glasgow, G12 8QQ, United Kingdom}
\email{Theresa.Ortscheidt@glasgow.ac.uk}

\author[T. S. Papazachariou]{Theodoros S. Papazachariou}
\address{Isaac Newton Institute, University of Cambridge, 20 Clarkson rd, Cambridge, CB3 0EH, United Kingdom}
\email{tsp35@cam.ac.uk}

\begin{document}

\maketitle

\begin{abstract}
    We study GIT stability of divisors in products of projective spaces. We first construct a finite set of one-parameter subgroups sufficient to determine the stability of the GIT quotient. In addition, we characterise all maximal orbits of not stable and strictly semistable pairs, as well as minimal closed orbits of strictly semistable pairs. This characterisation is applied to classify the GIT quotient of threefolds of bidegree (1,2) and bidegree (4,4) curves in a quadric surface, via singularities, which are in turn used to obtain an explicit description of the K-moduli space of family 2.25 of Fano threefolds, and the K-moduli wall-crossing of log Fano pairs.
\end{abstract}

\section{Introduction} \label{sec:intro}

Geometric invariant theory (GIT) is one of the primary tools in constructing moduli spaces, a problem which has been central in algebraic geometry in the last century. It has seen prominent use in the study of hypersurfaces \cite{Mumford_1977, Allcock_2002, Laza_2008}, complete intersections \cite{Avritzer_Lange_2000, hattori2022git} and divisors in smooth quadric surfaces \cite{Shah_1981, fedorchuk, Laza_OGrady_2018}. More recently, it has found great success in describing K-moduli spaces \cite{odaka_spotti_sun_2016, liu-xu,pap22} and K-moduli wall-crossing \cite{ascher2019wall, ascher2020kmoduli, Gallardo_2020, pap22}. Although GIT is a key tool for describing moduli spaces, explicit descriptions of GIT quotients remain hard, especially in higher dimensions. Recent advances have approached the GIT constructions computationally, developing mathematical algorithms and computer code to describe unstable elements in GIT quotients \cite{Gallardo_2018, bata458, pap22, theodoros_stylianos_papazachariou_2022, gallardo2023computation, gallardo_code1, gallardo_code2}. 

In this paper, we consider GIT quotients parametrising divisors of bidegree $(k,l)$ in products of projective space $\pr^m \times \pr^n$. Our framework can be automatised to perform computations for any pair $(m,n)$ and bidegree $(k,l)$. All the invariants and functions in this article have been implemented in software \cite{our_code}. Furthermore, this setting can be naturally extended to  divisors in arbitrary products of projective spaces.

Let $G \coloneqq \operatorname{SL}(m+1) \times \operatorname{SL}(n+1)$ and let $V_{k,l} \coloneqq |\mathcal{O}_{\pr^m \times \pr^n}(k,l)|$ be the linear system of divisors of bidegree $(k,l)$ in $\pr^m \times \pr^n$, which we will denote by $X_{k,l}$. We will study the GIT quotient $\pr V^* \sslash G$ computationally. Our first main result is that there exists a finite set $S_{m,n}^{k,l}$ of one-parameter subgroups of $G$, which is described combinatorially and can be computed via software, that determines the instability of any divisor. In particular:

\begin{theorem}[see Theorem \ref{thm: fundamental set}]
    A divisor $X_{k,l}$ is not stable if and only if there exist $g \in G, \lambda \in S_{m,n}^{k,l}$ destabilising $X_{k,l}$.
\end{theorem}

After we fix a coordinate system, each divisor $X_{k,l}$ can be determined by a homogeneous polynomial $F$ of degrees $k$ and $l$ in the coordinates of $\pr^m$ and $\pr^n$, respectively. This defines a set of monomials, namely those which appear with non-zero coefficients in $F$. If a divisor $X_{k,l}$ is not stable, then we can find a set of monomials $N^{\oplus}(\lambda)$ such that, in some coordinate system, the equation of $F$ is given by monomials in $N^{\oplus}(\lambda)$. A similar procedure follows for unstable divisors, where the relevant set of monomials is $N^{+}(\lambda)$. In more detail:

\begin{theorem}[see Theorem \ref{thm: main for divisors}]\label{thm:intro 2}
    A divisor $X_{k,l}$ is not stable (respectively unstable) if and only if there exists $g \in G$ such that the associated set of monomials of $g \cdot X_{k,l}$ is contained in $N^{\oplus}(\lambda)$ (respectively $N^+(\lambda)$) for some $\lambda \in S_{m,n}^{k,l}$.
\end{theorem}

Using the above results, we can determine whether a divisor $X_{k,l}$ is unstable or not stable. To determine whether the divisor is strictly semistable or stable, we use the Centroid Criterion. This is a polyhedral criterion that allows us to determine the stability of $X_{k,l}$, relying on the fact that each divisor $X_{k,l}$ defines a convex polytope $\operatorname{Conv}(X_{k,l})$. It stems from \cite[\S 7.2]{Mukai_2003} and \cite{Gallardo_2018}, which  used this criterion to study GIT stability of divisors in a single projective space. The theoretical framework used here extends this criterion to products of projective space. We also define a specific point $\mathcal{O}_{k,l}$, called the \emph{centroid} (for more details see Section \ref{sec:centroid}).

\begin{theorem}[see Theorem \ref{thm:centroid}]\label{thm: intro 3}
    A divisor $X_{k,l}$ is semistable (respectively stable) if and only if $\mathcal{O}_{k,l} \in \operatorname{Conv}(X_{k,l})$ (respectively $\mathcal{O}_{k,l} \in \operatorname{Int}(\operatorname{Conv}(C_{k,l}))$).
\end{theorem}

Determining strictly semistable divisors also allows us to determine the potential closed orbits associated with that divisor, and the GIT polystable divisors of bidegree $(k,l)$.

We should note that the above results naturally extend to GIT problems of divisors of bidegree $(k_1,\dots, k_l)$ in products of projective spaces $\mathbb{P}^{m_1}\times \dots \mathbb{P}^{m_l}$. Theorems \ref{thm:intro}, \ref{thm:intro 2} and \ref{thm: intro 3} generalise completely using the methods of proof presented in this paper. We choose to omit these descriptions and proofs in order to ease notation and improve readability.

We apply the above results in two different settings. First, we use the algorithm detailed in Section \ref{sec: computer implementation} and the implemented software to classify the GIT quotient of $(4,4)$-curves in $\pr^1 \times \pr^1$. Using our algorithm, we recover successfully previous results on the classification of semistable, polystable and stable orbits in this GIT quotient, due to Shah \cite[\S 4]{Shah_1981} and Laza--O'Grady \cite[Lemma 3.2, Proposition 3.3]{Laza_OGrady_2018}.

\begin{theorem}[see Theorem \ref{stable 4,4 curves}]
    Let $C$ be a $(4,4)$-curve in $\pr^1 \times \pr^1$. Then $C$ is stable if and only if one of the following holds:
    \begin{enumerate}
        \item $C$ is an irreducible, possibly singular curve, with singularities better than $\mathbf{X}_9$.
        \item $C$ is reducible, containing a $(1,0)$-ruling, and the residual $(3,4)$-curve $C'$ intersects the ruling in a unique point with multiplicity $\leq 3$ that is also a singular point of $C$, with singularities better than $\mathbf{J}_{2,0}$.
    \end{enumerate}
\end{theorem}

Here, a $\mathbf{J}_{2,0}$ singularity is a singularity with  normal form $x^3+bx^2y^2+y^6$ (c.f \cite[p. 95]{Arnold1976}). As a consequence of the above theorem and \cite[Theorem 1.1]{ascher2020kmoduli}, we obtain an explicit description of the K-moduli stack parametrising K-semistable log Fano pairs $\big(\pr^1 \times \pr^1, cC_{4,4}\big)$, for $0<c<1/8$.

We then consider a higher dimensional example, and apply it to obtain an explicit description of a K-moduli space for Fano threefolds. We use the algorithm detailed in Section \ref{sec: computer implementation} and the implemented software to classify the GIT quotient of $(1,2)$-divisors in $\pr^1 \times \pr^3$. Such smooth divisors are general members of family 2.25 of Fano threefolds in the Mori--Mukai taxonomy.

\begin{theorem}[See Theorem \ref{stable 1,2 divisors}, Propositions \ref{semistable 1,2 divisors} and \ref{polystable 1,2 divisors}]\label{thm:intro}
    Let $X$ be a $(1,2)$-divisor in $\pr^1\times \pr^3$. Then $X$ is GIT
    \begin{enumerate}
        \item stable if and only if $X$ is smooth;
        \item polystable if and only if $X$ is toric with four $\mathbf{A}_1$ singularities;
        \item strictly semistable if and only if $X$ is singular with one or two $\mathbf{A}_1$ singularities.
    \end{enumerate}
\end{theorem}

As a consequence of Theorem \ref{thm:intro}, we obtain a description of the K-moduli stack $\mathcal{M}^K_{2.25}$ parametrising K-semistable Fano varieties in family 2.25, and its corresponding K-moduli space ${M}^K_{2.25}$, that is different from the description that was obtained in \cite[\S 5]{pap22}. Let  $\mathcal{M}^{GIT}_{1,2}$ be the GIT moduli stack parametrising GIT semistable $(1,2)$-divisors in $\pr^1\times \pr^3$, and the GIT quotient ${M}^{GIT}_{1,2}$, which is classified in Theorem \ref{thm:intro}.

\begin{theorem}[See Theorem \ref{k-moduli 2.25}]
    There exists an isomorphism $\mathcal{M}^K_{2.25} \cong \mathcal{M}^{GIT}_{1,2}$. In particular, there exists an isomorphism ${M}^K_{2.25} \cong {M}^{GIT}_{1,2}$.
\end{theorem}

\renewcommand{\abstractname}{Acknowledgements}
\begin{abstract}
We would like to thank Ruadha{\'i} Dervan and Yuchen Liu for advice and useful comments. TP would like to thank the Isaac Newton Institute for Mathematical Sciences, Cambridge, for support and hospitality during the programme ``New equivariant methods in algebraic and differential geometry", where work on this paper was undertaken, while he was an INI Postdoctoral Research Fellow. IK and TAO were funded by a summer project fellowship associated with the Royal Society University Research Fellowship held by Ruadha{\'i} Dervan. TP was also funded by a postdoctoral fellowship associated with the aforementioned Royal Society University Research Fellowship. An important part of this paper was written during a visit by TP at the University of Glasgow. We would like to thank both Ruadha{\'i} Dervan and the UoG for the hospitality.
\end{abstract}

\section{Preliminaries}\label{sec: prelims}

Throughout this paper, we work over $\mathbb{C}$. Let $G \coloneqq \operatorname{SL}(m+1) \times \operatorname{SL}(n+1)$. We fix two positive integers $k,l$, and consider divisors of bidegree $(k,l)$ in $\pr^m \times \pr^n$. Note that the connected component of the identity in $\Aut(\pr^m \times \pr^n)$ is $\operatorname{PGL}(m+1)\times \operatorname{PGL}(n+1)$. Hence, $\Aut(\pr^m \times \pr^n) = \operatorname{PGL}(m+1)\times \operatorname{PGL}(n+1) \rtimes H$, where $H$ is a finite group. As such, there exists a natural $\operatorname{PGL}(m+1)\times \operatorname{PGL}(n+1)$-action on each divisor of bidegree $(k,l)$, and it suffices to only consider this action on each divisor.

Denoting the variables on $\pr^m$ by $x_0,\dots,x_m$ and the ones on $\pr^n$ by $y_0,\dots,y_n$, a $(k,l)$-divisor on $\pr^m \times \pr^n$ is defined by a bihomogeneous polynomial $F$ of bidegree $(k,l)$ in the variables $x_0,\dots,x_m,y_0,\dots,y_n$. That is, we have $F = \sum_I c_I x^I$, where $x^I \coloneqq x_0^{k_0} \, \cdots \,\, x_m^{k_m} y_0^{l_0} \, \cdots \,\, y_n^{l_n}$ for $I = (k_0, \dots, k_m, l_0, \dots, l_n) \in \Z^{m+n+2}$, such that $\sum_i k_i = k, \sum_j l_j = l$, and constants $c_I$. We define $\Xi_{k,l}$ to be the set of all monomials of bidegree $(k,l)$ in the variables $x_0,\dots,x_m,y_0,\dots,y_n$. Given a divisor $X_{k,l}$, its associated set of monomials is
$$
\operatorname{Supp}(X_{k,l}) \coloneqq \big\{ x^I\in \Xi_{k,l} \mid c_I \neq 0 \big\}.
$$

We consider the linear system  $V_{k,l} \coloneqq |\mathcal{O}_{\pr^m \times \pr^n}(k,l)|$ of $(k,l)$-divisors in $\pr^m \times \pr^n$ and we aim to study the GIT quotient $\pr V^* \sslash G$. This can be achieved by using the Hilbert-Mumford numerical criterion \cite[Theorem 2.1]{mumford}. To this end, we fix a maximal torus $T \subset G$ and a coordinate system on $\pr^m \times \pr^n$ such that $T$ is diagonal in $G$. Because a product of maximal tori is a maximal torus, we may assume without loss of generality, that $T$ is of the form $T = T_1 \times T_2$ for two maximal tori $T_1 \subset \operatorname{SL}(m+1), T_2 \subset \operatorname{SL}(n+1)$ that are diagonal. In this coordinate system, a one-parameter subgroup $\lambda: \mathbb{G}_m \to T$ is given by a pair of diagonal matrices
$$
\lambda(t) = \big ( \diag(t^{r_0}, \dots, t^{r_m}),\diag(t^{s_0}, \dots, t^{s_n}) \big),
$$
such that $r_i,s_j \in \Z$ for all $i=0,\dots,m$ and $j=0,\dots,n$, and $\sum_{i=0}^m r_i = \sum_{j=0}^n s_j = 0$. We call $\lambda$ \emph{normalised} if it is non-trivial, and we have $r_0 \ge \dots \ge r_m$ and $s_0 \ge \dots \ge s_n$. Note that the normalisation condition forces $r_0, s_0 \ge 0$ and $r_m, s_n \le 0$. We also note that every one-parameter subgroup of $T$ is conjugate to a normalised one-parameter subgroup.

Given a one-parameter subgroup $\lambda$ as above, the natural action of $\lambda$ on the coordinates of $\pr^m \times \pr^n$ induces an action on monomials in $\Xi_{k,l}$. Namely, given a monomial $x^I = x_0^{k_0} \, \cdots \,\, x_m^{k_m} y_0^{l_0} \, \cdots \,\, y_n^{l_n}$, we have
$$
\lambda(t) \cdot x^I = t^{\langle I, \lambda \rangle} x^I,
$$
where $\langle I, \lambda \rangle \coloneqq \sum_{i=0}^m r_i k_i + \sum_{j=0}^n s_j l_j$. This then extends to an action on the polynomial defining a divisor $X_{k,l}$ as follows:
$$
\lambda(t) \cdot X_{k,l} = \smashoperator[r]{\sum_{{x^I} \in \operatorname{Supp}(X_{k,l})}} c_It^{\langle I, \lambda \rangle} x^I.
$$
Note that the bilinear pairing $\langle -,- \rangle$ is the usual inner product on $\R^{m+n+2}$, restricted to $\Z^{m+n+2}$. We define the \emph{Hilbert-Mumford function} as
\begin{equation*}
    \mu(X_{k,l},\lambda) \coloneqq \min \Big \{ \langle I, \lambda \rangle \mid x^I \in \operatorname{Supp}(X_{k,l}) \Big \}.
\end{equation*}
This definition allows us to state the Hilbert-Mumford numerical criterion  \cite[Theorem 2.1]{mumford} as follows:

\begin{lemma} \label{lemma: hilbert-mumford}
    A divisor $X_{k,l}$ is not stable (respectively unstable) if and only if there exists $g \in G$ and a normalised one-parameter subgroup $\lambda$ of $T \subset G$ such that $\mu(g \cdot X_{k,l},\lambda) \ge 0$ (respectively $> 0$).
\end{lemma}
\begin{proof}
    Consider a divisor $X_{k,l}$ in $\pr^m \times \pr^n$ that is not stable. Then there exists a one-parameter subgroup $\theta$ in some maximal torus $T'$ of $G$, possibly different from $T$, such that $\mu(X_{k,l}, \theta )\geq 0$. Since all maximal tori of $G$ are conjugate, it follows from \cite[Exercise 9.2(i)]{dolgachev_2003} that
    $$
    \mu(X_{k,l}, \theta) = \mu(g\cdot X_{k,l}, g \theta g^{-1}),
    $$
    for all $g \in G$. Thus, there exists $g_0 \in G$ such that $\lambda \coloneqq g_0 \theta g_0^{-1}$ is a normalised one-parameter subgroup in the diagonal torus $T$, and $g_0 \cdot X_{k,l}$ has coordinates such that $\mu(g_0 \cdot X_{k,l}, \lambda) \ge 0$. Similarly, if $X_{k,l}$ is unstable, the same argument works by replacing $\ge$ with $>$.
\end{proof}

\section{GIT of \texorpdfstring{$(k,l)$}{}-divisors in \texorpdfstring{$\pr^m \times \pr^n$}{}} \label{sec:GIT theorems}

We consider the GIT quotient $\pr V^* \sslash G$ of bidegree $(k,l)$-divisors in $\pr^m \times \pr^n$, where $G= \operatorname{SL}(m+1) \times \operatorname{SL}(n+1)$. A few of these GIT quotients have been described explicitly, with particular emphasis given to divisors in a smooth quadric surface $\pr^1\times \pr^1$ (c.f. \cite{fedorchuk}). In this section, we will extend the already existing low-dimensional setting to a general computational approach, that will allow us to classify such GIT quotients in arbitrary products of projective spaces. We fix a maximal torus $T \subset G$ and a coordinate system on $\pr^m \times \pr^n$ such that $T$ is diagonal in $G$. 

\subsection{Stability Conditions}
In this section, we establish that only a finite set of normalised one-parameter subgroups in $N = \operatorname{Hom}_{\Z}(\mathbb{G}_m, T)$ is required to characterise instability of divisors of bidegree $(k,l)$ on $\pr^m \times \pr^n$. We will be following a similar approach to \cite{Gallardo_2018, pap22}.

\begin{definition} \label{defn: new fundamental set}
    The \emph{fundamental set $S_{m,n}^{k,l}$ of one-parameter subgroups} is given by all $\lambda(t) = (\diag(t^{r_0}, \dots, t^{r_m}),\diag(t^{s_0}, \dots,  t^{s_n})) \in T$ such that 
    $$
    (r_0, \dots, r_m, s_0, \dots, s_n) = c(\rho_0, \dots, \rho_m, \sigma_0, \dots, \sigma_n) \in \Z^{m+n+2},
    $$
    subject to the following conditions:
    \begin{enumerate}
        \item $\rho_i = \frac{a_i}{b_i} \in \Q$ and $\sigma_j = \frac{c_j}{d_j} \in \Q$ with $\gcd(a_i,b_i)=1$ and $\gcd(c_j,d_j)=1$ for all $i=0,\dots,m$ and $j=0,\dots,n$;
        \item $c = \operatorname{lcm}(b_0, \dots, b_m, d_0, \dots, d_n)$;
        \item $\rho_0 \ge \cdots \ge \rho_m$ and $\sigma_0 \ge \cdots \ge \sigma_n$, such that $\sum_{i=0}^m \rho_i = 0$ and $\sum_{j=0}^n \sigma_j = 0$;
        \item
        $(\rho_0, \dots, \rho_m, \sigma_0, \dots, \sigma_n)$ is the unique non-trivial solution of a consistent linear system of $m+n-1$ equations chosen from the set
       \end{enumerate}
\begin{align} \label{eqn: fundamental set}
        \operatorname{Eq}_{m,n}^{k,l} \coloneqq
                \Bigg\{ \sum_{i=0}^m w_i \rho_i + \sum_{j=0}^m v_j \sigma_j = 0 \mid w_i \in [-k,k] &\cap \Z \, \text{ for all } i, \sum_{i=0}^m w_i = 0, \\
                v_j \in [-l,l]  &\cap \Z \, \text{ for all } j,\sum_{j=0}^n v_j = 0 \Bigg\},
                \notag
\end{align}
    after fixing $\rho_0 = 1$ or $\sigma_0 = 1$.
    We note that the set $S_{m,n}^{k,l}$ is finite because $\operatorname{Eq}_{m,n}^{k,l}$ is finite.
\end{definition}

\begin{theorem} \label{thm: fundamental set}
    A divisor $X_{k,l}$  is not stable (respectively unstable) if and only if there exist $g \in G, \lambda \in S_{m,n}^{k,l}$ satisfying
    $$
    \mu(g \cdot X_{k,l}, \lambda) \ge 0 \quad \text{(respectively  $>0$)}.
    $$
\end{theorem}
\begin{proof}
    Suppose $X_{k,l}$ in $\pr^m \times \pr^n$ is not stable (respectively unstable). By Lemma \ref{lemma: hilbert-mumford}, there exist $g \in G$ and a normalised one-parameter subgroup $\lambda \in N$ such that $\mu(g \cdot X_{k,l}, \lambda) \ge 0$ (respectively  $>0$). We will now show that it is sufficient to take $\lambda$ to be in the finite set $S_{m,n}^{k,l}$.

    Extending to rational values, in the coordinates induced by $T$, normalised one-parameter subgroups are of the form $\lambda(t) = (\diag(t^{\rho_0}, \dots, t^{\rho_m}),\diag(t^{\sigma_0}, \dots,  t^{\sigma_n})) \in T$ with $(\rho_0, \dots, \rho_m, \sigma_0, \dots, \sigma_n) \in \Q^{m+n+2}$ subject to the conditions $\sum_{i=0}^m \rho_i = 0$ such that $\rho_i \ge \rho_{i+1}$ for $i = 0, \dots, m-1$, and $\sum_{j=0}^n \sigma_j = 0$ such that $\sigma_j \ge \sigma_{j+1}$ for $j = 0, \dots, n-1$. Geometrically, the set $\Gamma$ of all points in $\Q^{m+n+2}$ satisfying these conditions is the union of two convex cones $\Gamma_1$ and $\Gamma_2$ that intersect only at the origin.
    We may further impose the conditions $\rho_0 \le 1$ and $\sigma_0 \le 1$ without loss of generality. This amounts to intersecting $\Gamma$ with the two half-spaces $\rho_0 \le 1$ and $\sigma_0 \le 1$, which results in a set $\Delta$ that is a union of the two simplices $\Delta_1 = \Gamma_1 \cap \{ \rho_0 \le 1 \}$ and $\Delta_2 = \Gamma_2 \cap \{ \sigma_0 \le 1 \}$ of dimensions $m$ and $n$ respectively.
    With a slight abuse of notation, we will write $\lambda = (\rho_0, \dots, \rho_m, \sigma_0, \dots, \sigma_n) \in \Delta$ for the corresponding one-parameter subgroup $\lambda$. 
    
   Fixing a divisor $X_{k,l}$, the function $\mu(X_{k,l}, -): \Q^{m+n+2} \to \Q$ is continuous and piecewise linear. Furthermore, its restriction on $\Delta$, which is a compact subset, will attain a minimum and a maximum.
    We note that we only need to consider the one-parameter subgroups that correspond to the critical points of $\mu(X_{k,l}, -)$ restricted to $\Delta$, as these are sufficient to determine the sign of $\mu(X_{k,l}, \lambda)$, as $\lambda$ ranges through all possible values in $\Delta$.     The points on which $\mu(X_{k,l}, -)$ fails to be linear are precisely the points $\lambda \in \Delta$ such that $\langle I, \lambda \rangle = \langle I', \lambda \rangle$, where $I,I'$ correspond to a pair of distinct monomials $x^I,x^{I'} \in \operatorname{Supp}(X_{k,l})$. For each $(k,l)$, there are only finitely many monomials in $\Xi_{k,l}$, and therefore $\mu(X_{k,l}, -)$ has only a finite number of critical points on $\Delta$. 
    
    As $\langle -,- \rangle$ is bilinear, the condition $\langle I, \lambda \rangle = \langle I', \lambda \rangle$ is equivalent to $\langle I - I', \lambda \rangle = 0$. Writing $I = (k_0, \dots, k_m, l_0, \dots, l_n)$ and $I' = (k'_0, \dots, k'_m, l'_0, \dots, l'_n)$ with $\sum_{i=0}^m k_i = \sum_{i=0}^m k'_i = k$ and $\sum_{j=0}^n l_j = \sum_{j=0}^n l'_j = l$, the equation $\langle I - I', \lambda \rangle = 0$ becomes
    $$
    \sum_{i=0}^m (k_i - k'_i)\rho_i + \sum_{j=0}^n (l_j -l'_j)\sigma_j = 0.
    $$
    Making the substitutions $w_i \coloneqq k_i - k'_i$ and $v_j \coloneqq l_j - l'_j$ yields the set of equations $\operatorname{Eq}_{m,n}^{k,l}$ in \eqref{eqn: fundamental set}. Indeed, since we have $k_i,k'_i \in \{ 0, \dots, k\}$, we obtain $w_i \in \{-k, \dots, k \}$ and also $\sum_{i=0}^m w_i = \sum_{i=0}^m k_i - \sum_{i=0}^m k'_i = k - k = 0$. Similarly for the $v_j$.

    Each choice of the coefficients $w_i, v_j$ describes a hyperplane $H \subset \Q^{m+n+2}$. Upon fixing $\rho_0 = 1$ or $\sigma_0 = 1$, the intersection of $H$ with $\Delta$ is a union $\Delta'_1 \cup \Delta'_2 = (H\cap \Delta_1) \cup (H \cap \Delta_2)$, where $\Delta'_1$ and $\Delta'_2$ are simplices of dimensions $m-1$ and $n-1$ respectively. The critical points of $\mu(X_{k,l}, -)$ occur either on $\partial \Delta$ or on $\Delta'_1 \cup \Delta'_2$. The boundary $\partial \Delta$ consists of points $\lambda = (\rho_0, \dots, \rho_m, \sigma_0, \dots, \sigma_n) \in \Delta$ such that $\rho_i = \rho_{i+1}$ or $\sigma_j = \sigma_{j+1}$ for some $i$ or $j$. But these points may also be extracted from $\operatorname{Eq}_{m,n}^{k,l}$ by a suitable choice of $w_i$ and $v_j$; e.g. we can always choose the $w_i$ and $v_j$ so that $0 = \sum_{i=0}^m w_i \rho_i + \sum_{j=0}^n v_j \sigma_j = \rho_i - \rho_{i+1}$. The intersection of $m+n-1$ hyperplanes chosen from Equations \eqref{eqn: fundamental set} completely specifies a critical point in $\Q^{m+n+2}$. The set of all such points correspond to the points $(\rho_0, \dots, \rho_m, \sigma_0, \dots, \sigma_n)$ as described in Definition \ref{defn: new fundamental set}. Multiplying by a suitable constant gives the integral points specifying the normalised one-parameter subgroups in the set $S_{m,n}^{k,l}$.
\end{proof}

Theorem \ref{thm: fundamental set} allows us to completely characterise instability of $(k,l)$-divisors by only considering one-parameter subgroups in the finite set $S_{m,n}^{k,l}$.

\begin{definition} \label{def: new monomial sets}
Let $\lambda$ be a normalised one-parameter subgroup. A non-empty subset $A \subset \Xi_{k,l}$ is called \emph{maximal (semi-)destabilised with respect to $\lambda$} if the following conditions hold:
\begin{enumerate}
    \item for all $x^I\in A$, $\langle I,\lambda \rangle >0$ ($\geq 0$, respectively);
    \item if there is another subset $B \subset \Xi_{k,l}$ such that $A\subset B$, and for all $x^J\in B$ the inequality $\langle J,\lambda \rangle >0$ ($\geq 0$, respectively) holds, then $A=B$.
\end{enumerate}

\end{definition}

It is not hard to see that given a normalised one-parameter subgroup $\lambda$, the maximal (semi-)destabilised sets with respect to $\lambda$ are given by
    \begin{align*}
            N^+(\lambda) &\coloneqq \big\{ x^I \in \Xi_{k,l} \mid \langle I, \lambda \rangle > 0 \big\},\\
            N^{\oplus}(\lambda) &\coloneqq \big\{ x^I \in \Xi_{k,l} \mid \langle I, \lambda \rangle \geq 0 \big\}.
    \intertext{We also define the \emph{annihilator} subset}
    \operatorname{Ann}(\lambda) &\coloneqq \big\{ x^I \in \Xi_{k,l} \mid \langle I, \lambda \rangle = 0 \big\} \subset N^{\oplus}(\lambda).
    \end{align*}

We are now in a position to state our main results:

\begin{theorem} \label{thm: main for divisors}
    A divisor $X_{k,l}$ in $\pr^m \times \pr^n$ is not stable (respectively unstable) if and only if there exists $g \in G$ such that the associated set of monomials of $g \cdot X_{k,l}$ is contained in $N^{\oplus}(\lambda)$ (respectively $N^+(\lambda)$) for some $\lambda \in S_{m,n}^{k,l}$.
\end{theorem}
\begin{proof}
    By Theorem \ref{thm: fundamental set}, $X_{k,l}$ is not stable if and only if there exist $g \in G,\lambda \in S_{m,n}^{k,l}$ such that
    $$
    \mu(g \cdot X_{k,l}, \lambda) = \min \Big \{ \langle I, \lambda \rangle \mid x^I \in \operatorname{Supp}(g \cdot X_{k,l}) \Big \} \geq 0.
    $$
    This implies that $\langle I, \lambda \rangle \geq 0$ for all $x^I \in \operatorname{Supp}(g \cdot X_{k,l})$, hence $\operatorname{Supp}(g \cdot X_{k,l}) \subset N^{\oplus}(\lambda)$. Similarly, $X_{k,l}$ is unstable if and only if $\langle I, \lambda \rangle > 0$ for all $x^I \in \operatorname{Supp}(g \cdot X_{k,l})$, which implies $\operatorname{Supp}(g \cdot X_{k,l}) \subset N^+(\lambda)$. Choosing the maximal sets $N^{\oplus}(\lambda)$ under the containment order, where $\lambda \in S^{k,l}_{m,n}$, we obtain families of divisors whose coefficients belong to maximal destabilised sets.
\end{proof}

In light of Theorem \ref{thm: main for divisors}, we see that to describe the families of not stable (respectively unstable) divisors, it suffices to consider sets $N^{\oplus}(\lambda)$ (respectively $N^+(\lambda)$) that are maximal with respect to the containment order of sets, for $\lambda \in S_{m,n}^{k,l}$.

\begin{proposition}
 \label{thm: polystable divisors}
    If a divisor $X_{k,l}$ in $\pr^m \times \pr^n$ belongs to a closed strictly semistable orbit (i.e. a strictly polystable orbit), then there exist $g \in G, \lambda \in S_{m,n}^{k,l}$ such that $\operatorname{Supp}(g \cdot X_{k,l}) = \operatorname{Ann}(\lambda)$.
\end{proposition} 
\begin{proof}
    Let $X_{k,l}$ be a strictly semistable divisor representing a closed orbit. Then, by \cite[Remark 8.1(5)]{dolgachev_2003} its stabiliser subgroup $\operatorname{Stab}(X_{k,l}) \subset G$ is infinite. This implies that there exists a one-parameter subgroup $\lambda$ whose image lies in $\operatorname{Stab}(X_{k,l})$, so that $\lim_{t \to 0} (\lambda(t) \cdot X_{k,l}) = X_{k,l}$. Hence, $\langle I, \lambda \rangle = 0$ for all $x^I \in \operatorname{Supp}(X_{k,l})$. By a suitable choice of a coordinate system, Theorem \ref{thm: main for divisors} gives a $g \in G$ and a $\lambda \in S_{m,n}^{k,l}$ such that $\operatorname{Supp}(g \cdot X_{k,l}) = \operatorname{Ann}(\lambda)$.
\end{proof}

\subsection{The Centroid Criterion} \label{sec:centroid}

The Centroid Criterion is an efficient method for determining whether a given divisor in $\pr^m \times \pr^n$ is stable (or semistable), without the need for one-parameter subgroups. Our approach here is similar to \cite[\S 7.2]{Mukai_2003}, \cite[Lemma 1.5]{Gallardo_2018} and \cite[Theorem 3.18]{pap22}, which  used this criterion to study GIT stability of divisors and complete intersections in a single projective space. 

We fix two positive integers $k$ and $l$, and define a map $\xi : \Xi_{k,l} \to \Z^{m+n}$ such that for any monomial $x^I \in \Xi_{k,l}$ with $I=(k_0,\dots,k_m,l_0,\dots,l_n)$, we have $\xi(x^I) = (k_0,\dots,k_{m-1},l_0,\dots,l_{n-1})$. 
Since the terms $k_m$ and $l_n$ are completely determined by the conditions $\sum_{i=0}^m k_i = k$ and $\sum_{j=0}^n l_j = l$, the map $\xi$ uniquely identifies each monomial $x^I \in \Xi_{k,l}$ with a point in $\Z^{m+n}$. This gives an integer lattice in $\R^{m+n}$, and for a given divisor $X_{k,l}$ in $\pr^m \times \pr^n$, the set $\xi(\operatorname{Supp}(X_{k,l}))$ is a subset of this lattice. To ease notation, we will denote this subset by $\xi(X_{k,l})$. Finally, we define $\operatorname{Conv}(X_{k,l}$) to be the convex hull of the set $\xi(X_{k,l})$. As $\xi(X_{k,l})$ is a finite set of points in $\R^{m+n}$, the convex hull $\operatorname{Conv}(X_{k,l})$ is automatically compact, and is simply a convex polytope. We denote the interior of this polytope by $\operatorname{Int(Conv}(X_{k,l}))$. We define the \emph{centroid} as 
$$
\mathcal{O}_{k,l} \coloneqq \biggl(\underbrace{\frac{k}{m+1},\dots,\frac{k}{m+1}}_{m},\underbrace{\frac{l}{n+1},\dots,\frac{l}{n+1}}_{n}  \biggr) \in \Q^{m+n}.
$$

\begin{theorem}[Centroid Criterion] \label{thm:centroid}
    A divisor $X_{k,l}$ is semistable (respectively stable) if and only if $\mathcal{O}_{k,l} \in \operatorname{Conv}(X_{k,l})$ (respectively $\mathcal{O}_{k,l} \in \operatorname{Int}(\operatorname{Conv}(X_{k,l}))$).
\end{theorem}
\begin{proof}
    First suppose $\mathcal{O}_{k,l} \not \in \operatorname{Conv}(X_{k,l})$. Since $\operatorname{Conv}(X_{k,l})$ is a convex polytope, by the Hyperplane Separation Theorem \cite[Examples 2.19,2.20]{convex}, there exists an affine function $\Phi : \mathbb{R}^{m+n} \to \mathbb{R}$ that is positive at points in $\operatorname{Conv}(X_{k,l})$ and zero at $\mathcal{O}_{k,l}$. Moreover, $\Phi$ has the form
    $$
    \Phi(x_0,\dots,x_{m-1},y_0,\dots,y_{n-1}) = \sum_{i=0}^{m-1} \alpha_i \left(x_i - \frac{k}{m+1} \right) + \sum_{j=0}^{n-1} \beta_j \left(y_j - \frac{l}{n+1} \right),
    $$
    where the $\alpha_i$ and $\beta_j$ are constants. Writing $A \coloneqq \sum_{i=0}^{m-1} \alpha_i$ and $B \coloneqq \sum_{j=0}^{n-1} \beta_j$, we have
    $$
    \Phi(x_0,\dots,x_{m-1},y_0,\dots,y_{n-1}) = \sum_{i=0}^{m-1} \alpha_i x_i + \sum_{j=0}^{n-1}\beta_j y_j- \frac{kA}{m+1} - \frac{lB}{n+1}.
    $$
    As the vertices of the polygon $\operatorname{Conv}(X_{k,l})$ have integer coordinates, we may further assume that the coefficients $\alpha_i,\beta_j$ are rational. But, upon multiplying through by a suitable positive integer, we can obtain integer coefficients. Since this results in an affine function that is simply a positive scalar multiple of $\Phi$, it will have the same properties as $\Phi$. Thus, we may assume that all the coefficients $\alpha_i$ and $\beta_j$ are integers without loss of generality.

    We define $r_i = (m+1)(n+1)\alpha_i -(n+1)A$ for $i=0,\dots,m-1$, and $r_m = -(n+1)A$. Similarly, we define $s_j = (m+1)(n+1)\beta_j -(m+1)B$ for $j=0,\dots,n-1$, and $s_n = -(m+1)B$. By construction, we have $r_i \in \Z$ for all $i=0,\dots,m$ with $\sum_{i=0}^m r_i = 0$ and $s_j \in \Z$ for all $j=0,\dots,n$ with $\sum_{j=0}^n s_j =0$. Thus, we can form a one-parameter subgroup $\lambda$ given by $\lambda(t) = (\diag(t^{r_0},\dots,t^{r_m}),\diag(t^{s_0},\dots,t^{s_n}))$.

    For $I=(k_0,\dots,k_m,l_0,\dots,l_n) = (k_0,\dots,k_{m-1},k-\sum_{i=0}^{m-1}k_i,l_0,\dots,l_{n-1},l-\sum_{j=0}^{n-1}l_j) $ associated to some monomial $x^I \in \operatorname{Supp}(X_{k,l})$, we have
    \begin{align*}
        \langle I, \lambda \rangle &= \sum_{i=0}^{m} k_ir_i + \sum_{j=0}^{n} l_js_j \\
        &= (m+1)(n+1)\Bigg[ \sum_{i=0}^{m-1} k_i\alpha_i + \sum_{j=0}^{n-1} l_j\beta_j  -\frac{A}{m+1}\left( \sum_{i=0}^{m-1}k_i + k_m \right) -\frac{B}{n+1}\left( \sum_{j=0}^{n-1}l_j + l_n \right) \Bigg] \\
        &= (m+1)(n+1)\Bigg[ \sum_{i=0}^{m-1} k_i\alpha_i + \sum_{j=0}^{n-1} l_j\beta_j  -\frac{Ak}{m+1} -\frac{Bl}{n+1} \Bigg] \\
        &= (m+1)(n+1)\Phi(k_0,\dots,k_{m-1},l_0,\dots,l_{n-1}).
    \end{align*}
    But $\Phi$ is positive at points in $\operatorname{Conv}(X_{k,l})$, and so we have $\langle I, \lambda \rangle > 0 $ for all $x^I \in \operatorname{Supp}(X_{k,l})$. Hence,
    $$
    \mu(X_{k,l},\lambda) = \min \Big\{ \langle I, \lambda \rangle \mid x^I \in \operatorname{Supp}(X_{k,l}) \Big\} > 0,
    $$ 
    and therefore $X_{k,l}$ is unstable. We have shown that if $\mathcal{O}_{k,l} \not \in \operatorname{Conv}(X_{k,l})$ then $X_{k,l}$ is unstable. The contrapositive then shows that if $X_{k,l}$ is semistable, then $\mathcal{O}_{k,l} \in \operatorname{Conv}(X_{k,l})$. The argument for the case when $X_{k,l}$ is stable is completely analogous; in the above steps, we swap $\operatorname{Conv}(X_{k,l})$ with $\operatorname{Int}(\operatorname{Conv}(X_{k,l}))$ and $>$ with $\ge$.

    Conversely, suppose $X_{k,l}$ is unstable. Then, by Lemma \ref{lemma: hilbert-mumford}, there exists a one-parameter subgroup $\lambda(t) = (\diag(t^{r_0},\dots,t^{r_m}),\diag(t^{s_0},\dots,t^{s_n}))$ such that
    \begin{align*}
        \mu(X_{k,l},\lambda) &=\min \Big \{ \langle I, \lambda \rangle \mid x^I \in \operatorname{Supp}(X_{k,l}) \Big \} \\
        &= \smashoperator[l]{\min_{\prod_{i,j}x_i^{k_i}y_j^{l_j}\in \Xi_{k,l}}} \Bigg\{\sum_{i=0}^m k_ir_i + \sum_{j=0}^n l_js_j  \Bigg\} \\
        &= \smashoperator[l]{\min_{\prod_{i,j}x_i^{k_i}y_j^{l_j}\in \Xi_{k,l}}} \Bigg\{\sum_{i=0}^{m-1} k_ir_i + \sum_{j=0}^{n-1} l_js_j - \left(k-\sum_{i=0}^{m-1} k_i\right)\sum_{i=0}^{m-1}r_i - \left(l-\sum_{j=0}^{n-1} l_i\right)\sum_{j=0}^{n-1}s_j  \Bigg\} \\
        &>0,
    \end{align*} 
    where we used that $r_m = -\sum_{i=0}^{m-1}r_i$ and $s_n = -\sum_{j=0}^{n-1}s_j$. Let $\Psi: \R^{m+n} \to \R$ be the affine function given by
    $$
    \Psi(x_0,\dots,x_{m-1},y_0,\dots,y_{n-1}) = \sum^{m-1}_{i=0} x_ir_i + \sum^{n-1}_{j=0} y_js_j- \Bigg (k-\sum_{i=0}^{m-1} x_i \Bigg )\sum_{i=0}^{m-1}r_i - \Bigg (l-\sum_{j=0}^{n-1} y_j \Bigg)\sum_{j=0}^{n-1}s_j.$$ 
    Then $\Psi$ is positive at all points in $\xi(X_{k,l})$. By convexity, we also have that $\Psi|_{\operatorname{Conv}(X_{k,l})} >0 $. On the other hand, 
    \begin{align*}
    \Psi(\mathcal{O}_{k,l}) &= \frac{k}{m+1}\sum_{i=0}^{m-1} r_i +\frac{l}{n+1}\sum_{j=0}^{n-1}s_j - \Bigg(k-\frac{mk}{m+1}\Bigg)\sum_{i=0}^{m-1}r_i - \Bigg(l-\frac{nl}{n+1}\Bigg)\sum_{j=0}^{n-1}s_j \\
    & = \frac{k}{m+1}\sum_{i=0}^{m-1} r_i +\frac{l}{n+1}\sum_{j=0}^{n-1}s_j - \frac{k}{m+1}\sum_{i=0}^{m-1}r_i - \frac{l}{n+1}\sum_{j=0}^{n-1}s_j \\ &=0.
    \end{align*}
    Therefore, $\mathcal{O}_{k,l} \not\in \operatorname{Conv}(X_{k,l})$. We have shown that if $X_{k,l}$ is unstable, then $\mathcal{O}_{k,l} \not \in \operatorname{Conv}(X_{k,l})$. The contrapositive then shows that if $\mathcal{O}_{k,l} \in \operatorname{Conv}(X_{k,l})$, then $X_{k,l}$ is semistable. Similarly, if $X_{k,l}$ is not stable, then a slight modification to the above argument will show that $\mathcal{O}_{k,l}$ at worst lies on the boundary of $\operatorname{Conv}(X_{k,l})$. Therefore, if $\mathcal{O}_{k,l} \in \operatorname{Int}(\operatorname{Conv}(X_{k,l}))$, then $X_{k,l}$ is stable.
\end{proof}

\begin{remark}
    We should note that all the results of Sections \ref{sec:GIT theorems} and \ref{sec:centroid} generalise naturally to GIT problems of divisors of bidegree $(k_1,\dots, k_l)$ in products of projective spaces $\pr^{m_1} \times \dots \times \pr^{m_l}$. In particular, Theorems \ref{thm: fundamental set}, \ref{thm: main for divisors}, \ref{thm:centroid} and Proposition \ref{thm: polystable divisors} generalise completely using the methods of proof presented in this paper. We choose to omit these descriptions and proofs in order to ease notation and improve readability. We do note that the computer software \cite{our_code} based on the results of this paper has been implemented with this level of generality.
\end{remark}

\subsection{Computer Implementation}\label{sec: computer implementation}

The following gives an algorithm for determining the stability and instability of families of divisors in $\pr^m \times \pr^n$ of some given bidegree $(k,l)$:

\begin{enumerate}[label=\arabic*.]
    \item We begin by computing the fundamental set $S_{m,n}^{k,l}$ of one-parameter subgroups using Definition \ref{defn: new fundamental set}.
    \item For each $\lambda \in S_{m,n}^{k,l}$, we compute the sets $N^{\oplus}(\lambda)$ and $N^+(\lambda)$, as defined in Definition \ref{def: new monomial sets}. 
    \item We then find all the maximal sets with respect to the containment order of sets, and discard the rest. Thus, by Theorem \ref{thm: main for divisors} and Proposition \ref{thm: polystable divisors}, these will describe families of unstable ($N^+(\lambda)$) divisors and not stable ($N^{\oplus}(\lambda)$) divisors.
    \item On each maximal set $N^{\oplus}(\lambda)$, we apply the Centroid Criterion (c.f. Theorem \ref{thm:centroid}) to determine whether the family is a family of strictly semistable divisors. For each family satisfying the Centroid Criterion, we compute the annihilator $\operatorname{Ann}(\lambda)$ defined in Definition \ref{def: new monomial sets}. By Proposition \ref{thm: polystable divisors}, these families represent strictly polystable orbits.
    \item We now classify the families of divisors we have obtained by their singularities, and hence find the stable, polystable and semistable divisors by process of elimination.
\end{enumerate}

\section{GIT of (4,4)-Curves in \texorpdfstring{$\pr^1\times \pr^1$}{}}

We will apply the GIT algorithm detailed above (Section \ref{sec: computer implementation}) to classify the GIT quotient of $(4,4)$-curves in $\mathbb{P}^1_{X,Y}\times \mathbb{P}^1_{Z,W}$. Let $G \coloneqq \operatorname{SL}(2) \times \operatorname{SL}(2)$, and $V \coloneqq |\mathcal{O}_{\pr^1 \times \pr^1}(4,4)|$. We will study the GIT quotient $\mathbb{P} V^* \sslash G$ computationally. As mentioned before, this GIT quotient has been studied before in Shah \cite[\S 4]{Shah_1981} and Laza--O'Grady \cite[Lemma 3.2, Proposition 3.3]{Laza_OGrady_2018}. In this section, we recover the results of Shah and Laza--O'Grady using the different methods our algorithm allows, demonstrating its usefulness and efficiency. In addition, we provide a visual classification of GIT strictly semistable curves, using Theorem \ref{thm:centroid}, and we provide a detailed analysis of the singularity types of semistable, polystable and stable curves, which is not present in the aforementioned works.

\subsection{Algorithm Output}\label{sec: algorithm output}

We will denote a normalised one-parameter subgroup $\lambda$ in $G$ by $\lambda = (u,-u,v,-v)$. For convenience, we also introduce the notation $\Tilde{\lambda} = (v,-v,u,-u)$. By our algorithm (Definition \ref{defn: new fundamental set} and Theorem \ref{thm: fundamental set}), we obtain the fundamental set $S_{1,1}^{4,4}$, which has 13 elements. The relevant one-parameter subgroups in $S_{1,1}^{4,4}$ that give maximal (semi-)destabilised sets are given by
\begin{align*}
\lambda_0 &= (1, -1, 0, 0),\\
\lambda_1 &= (1, -1, 1, -1),\\
\lambda_2 &= (2, -2, 1, -1),\\
\lambda_3 &= (3, -3, 1, -1),\\
\lambda_4 &= (3, -3, 2, -2),
\end{align*}
and $\Tilde{\lambda}_i$ for $i=0,2,3,4$. For any positive integer $k$, monomials of bidegree $(k,k)$ in $\pr^1 \times \pr^1$ are symmetric in the sense that any pair of one-parameter subgroups $\lambda$ and $\Tilde{\lambda}$ will give projectively isomorphic maximal (semi-)destabilising sets. Therefore, we only need to consider the one-parameter subgroups $\lambda_i$ for $i=0,\dots,4$. In particular, notice that we obtain a different description of the set of one-parameter subgroups used in \cite[\S 3.1]{Laza_OGrady_2018}, using Definition \ref{defn: new fundamental set}. We should note, that the three one-parameter subgroups used in \cite[\S 3.1]{Laza_OGrady_2018} are contained in $S_{1,1}^{4,4}$. Overall, we have three maximal semi-destabilised sets coming from the one-parameter subgroups $\lambda_0, \lambda_1$ and $\lambda_2$, and two maximal destabilised sets coming from the one-parameter subgroups $\lambda_3$ and $\lambda_4$. These sets of monomials are listed in Table \ref{tab:unstable-nonstable}.

\begingroup
\renewcommand{\arraystretch}{1.1}
\begin{table}[H]
   \centering
   \begin{NiceTabular}{|c|l|l|l|l|l|}
     \hline
      Families & $N^{\oplus}(\lambda_0)$ & $N^{\oplus}(\lambda_1)$ & $N^{\oplus}(\lambda_2)$ & $N^{+}(\lambda_3)$ & $N^{+}(\lambda_4)$ \\ 
      \hline
      \Block{18-1}{\rotate monomials}
       & $X^4Z^4$ & $X^4Z^4$ & $X^4Z^4$ & $X^4Z^4$ & $X^4Z^4$ \\
       & $X^4Z^3W$ & $X^4Z^3W$ & $X^4Z^3W$ & $X^4Z^3W$ & $X^4Z^3W$ \\
       & $X^4Z^2W^2$ & $X^4Z^2W^2$ & $X^4Z^2W^2$ & $X^4Z^2W^2$ & $X^4Z^2W^2$ \\
       & $X^4ZW^3$ & $X^4ZW^3$ & $X^4ZW^3$ & $X^4ZW^3$ & $X^4ZW^3$ \\
       & $X^4W^4$ & $X^4W^4$ & $X^4W^4$ & $X^4W^4$ & $X^4W^4$ \\
       & $X^3YZ^4$ & $X^3YZ^4$ & $X^3YZ^4$ & $X^3YZ^4$ & $X^3YZ^4$ \\
       & $X^3YZ^3W$ & $X^3YZ^3W$ & $X^3YZ^3W$ & $X^3YZ^3W$ & $X^3YZ^3W$ \\
       & $X^3YZ^2W^2$ & $X^3YZ^2W^2$ & $X^3YZ^2W^2$ & $X^3YZ^2W^2$ & $X^3YZ^2W^2$ \\
       & $X^3YZW^3$ & $X^3YZW^3$ & $X^3YZW^3$ & $X^3YZW^3$ & $X^3YZW^3$ \\
       & $X^3YW^4$ & --- & $X^3YW^4$ & $X^3YW^4$ & --- \\
       & $X^2Y^2Z^4$ & $X^2Y^2Z^4$ & $X^2Y^2Z^4$ & $X^2Y^2Z^4$ & $X^2Y^2Z^4$ \\
       & $X^2Y^2Z^3W$ & $X^2Y^2Z^3W$ & $X^2Y^2Z^3W$ & $X^2Y^2Z^3W$ & $X^2Y^2Z^3W$ \\
       & $X^2Y^2Z^2W^2$ & $X^2Y^2Z^2W^2$ & $X^2Y^2Z^2W^2$ & --- & --- \\
       & $X^2Y^2ZW^3$ & --- & --- & ---& --- \\
       & $X^2Y^2W^4$ & --- & --- & --- & --- \\
       & --- & $XY^3Z^4$ & $XY^3Z^4$ & --- & $XY^3Z^4$ \\
       & --- &  $XY^3Z^3W$ & --- & --- & --- \\
       & --- &  $Y^4Z^4$ & --- & --- & --- \\
       \hline
   \end{NiceTabular}
   \caption{Unstable and not-stable families and their monomials.}
   \label{tab:unstable-nonstable}
\end{table} 
\endgroup

In particular, the Centroid Criterion \ref{thm:centroid} shows that the three not-stable families $N^{\oplus}(\lambda_0)$, $N^{\oplus}(\lambda_1)$ and $N^{\oplus}(\lambda_2)$ are strictly semistable, as illustrated in Figure \ref{fig:centroid_lambda2}.

\begin{figure}[h]
\centering
\includegraphics[scale=0.375]{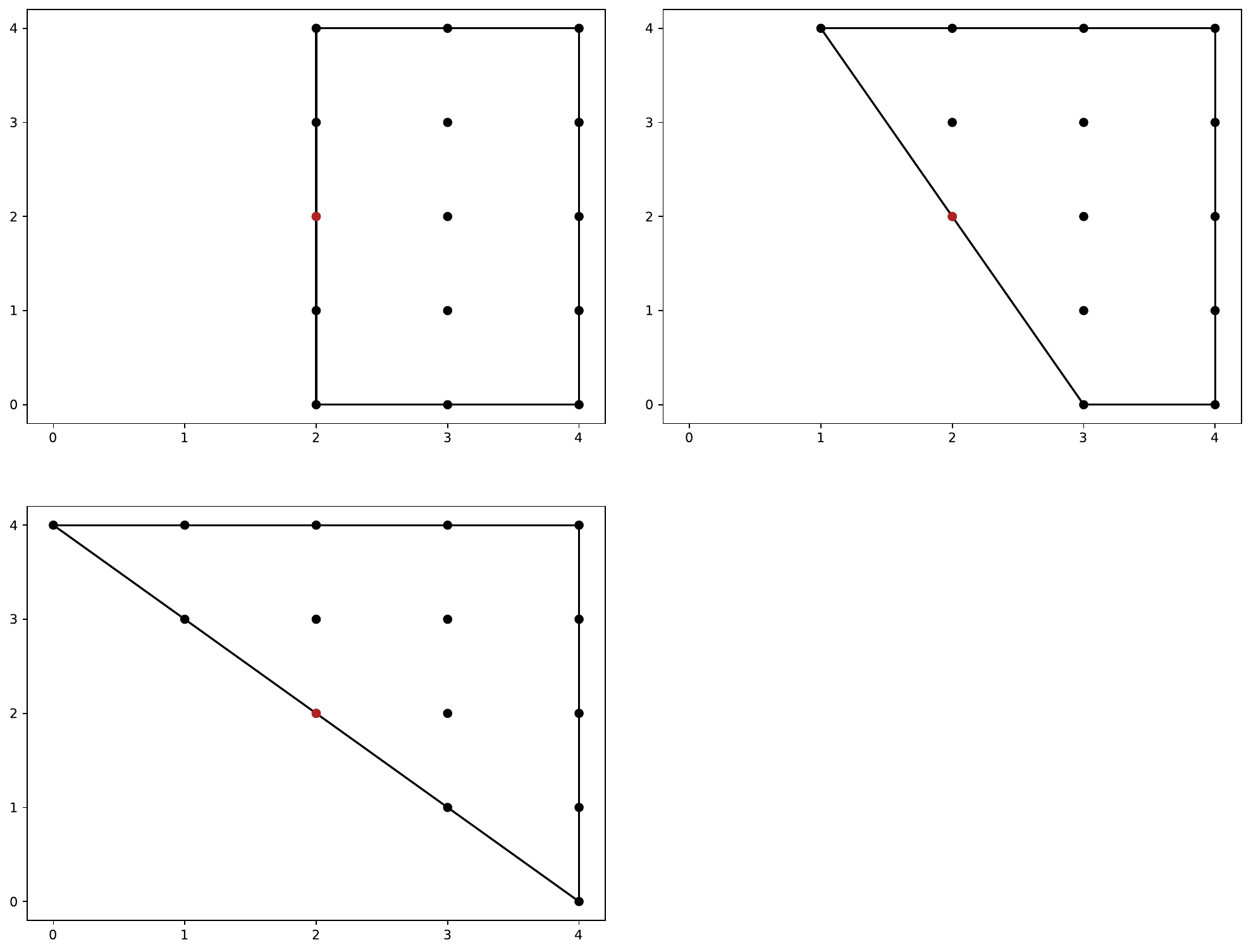}
\caption{Centroid criterion for $N^{\oplus}(\lambda_0)$ (top left), $N^{\oplus}(\lambda_2)$ (top right) and $N^{\oplus}(\lambda_1)$ (bottom left).}
\label{fig:centroid_lambda2}
\end{figure}

Here, the red point is the centroid $(2,2)$, and we can see that, in all three cases, it lies on the boundary of the polygon. The potential closed orbits, i.e. the monomials of zero weight given by $\operatorname{Ann}(\lambda_i)$ for $i=0,1,2$, are listed in Table \ref{tab:closed orbits}. 

\begingroup
\renewcommand{\arraystretch}{1.1}
\begin{table}[H]
   \centering
   \begin{NiceTabular}{|c|l|l|l|}
     \hline
      Families & $\operatorname{Ann}(\lambda_0)$ & $\operatorname{Ann}(\lambda_1)$ & $\operatorname{Ann}(\lambda_2)$ \\ 
      \hline
      \Block{5-1}{\rotate monomials}
      & $X^2Y^2Z^4$ & $Y^4Z^4$ & $XY^3Z^4$  \\
      & $X^2Y^2Z^3W$ &  $XY^3Z^3W$ & $X^3YW^4$  \\
      & $X^2Y^2Z^2W^2$ & $X^2Y^2Z^2W^2$ & $X^2Y^2Z^2W^2$  \\
      & $X^2Y^2ZW^3$ & $X^3YZW^3$ &  \\
      & $X^2Y^2W^4$ &  $X^4W^4$ & \\
      \hline
    \end{NiceTabular}
    \caption{Potential closed orbits.}
    \label{tab:closed orbits}
\end{table}
\endgroup

By contrast, Figure \ref{fig:centroid_lambda1} shows that the Centroid Criterion fails for $N^+(\lambda_3)$ and $N^+(\lambda_4)$, thus confirming that these describe unstable families. Figures \ref{fig:centroid_lambda2} and \ref{fig:centroid_lambda1} also provide visual confirmation that, up to projective isomorphism, our five sets capture all possible maximal (semi-)destabilised sets of monomials.

\begin{figure}[h]
\includegraphics[scale=0.375]{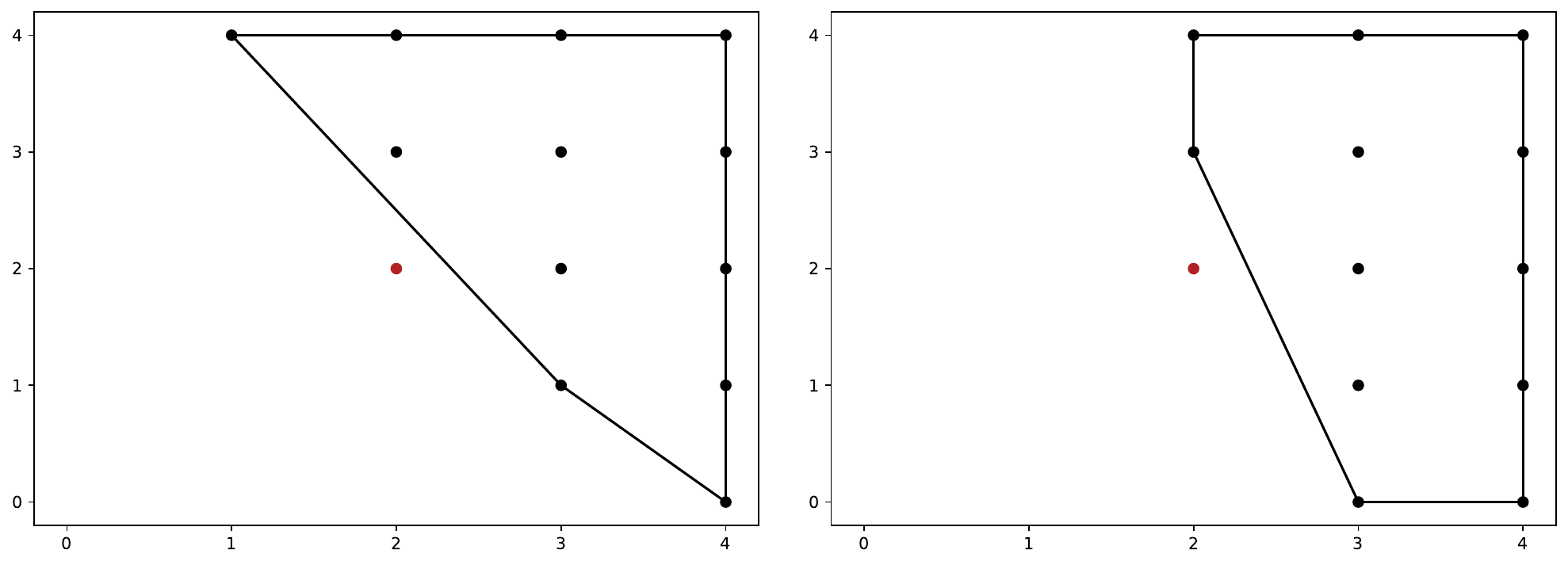}
\caption{Centroid criterion for $N^{+}(\lambda_4)$ (left) and $N^{+}(\lambda_3)$ (right).}
\label{fig:centroid_lambda1}
\end{figure}

\subsection{Some Preliminaries on Singularities.} \label{sec: prelims on singularity theory}

We will give some brief preliminaries on singularity theory, which will be used throughout the later chapters.

\begin{definition}[{\cite[p.88]{Arnold1976}}]
A class of singularities $T_2$ is \emph{adjacent} to a class $T_1$, and one writes $T_1 \leftarrow T_2$, if every germ of $f \in T_2$ can be locally deformed into a germ in $T_1$ by an arbitrary small deformation. We say that the singularity $T_2$ is \emph{worse} than $T_1$, or that $T_2$ is a \emph{degeneration} of $T_1$.
\end{definition}

The degenerations of the isolated singularities that appear in a curve of bidegree $(k,l)$ in $\mathbb{P}^1\times\mathbb{P}^1$ are described in Figure \ref{fig:dynkin} (for details, see \cite[p.92, 93, 95]{Arnold1976} and \cite[\S 13]{Arnol_d_1975}). The above theory considers only local deformations of singularities. When we study degenerations in the GIT quotient, we are interested in global deformations. Thankfully, in the particular case of study of this paper, any local deformation of isolated singularities is induced by a global deformation.  

\begin{figure}[h]
    \centering
    \begin{tikzcd}
    \mathbf{A} & \mathbf{D}\arrow[l] & \mathbf{E}_6\arrow[l]& \mathbf{E}_7\arrow[l]& \mathbf{E}_8\arrow[l]\\
    & &\mathbf{(P)} \arrow[u] &\mathbf{(X)} \arrow[u]& \mathbf{(J)} \arrow[u]
    \end{tikzcd}
    \caption{Some degenerations of germs of isolated singularities.}
    \label{fig:dynkin}
\end{figure}
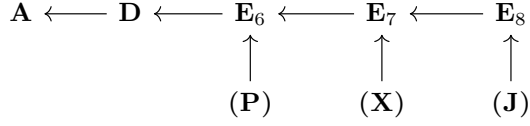

\subsection{Stability Analysis}\label{sec: stability analysis}
We first tackle unstable curves.

\begin{proposition}[Unstable curves]\label{unstable characterisation}
    Let $C$ be a $(4,4)$-curve in $\mathbb{P}^1\times\mathbb{P}^1$. Then $C$ is unstable if and only if one of the following holds:
    \begin{enumerate}
        \item $C$ contains a double ruling, such that the residual $(2,4)$-curve $C'$ is smooth and intersects the double ruling in two points.
        \item $C$ contains a $(1,0)$-ruling, such that the residual $(3,4)$-curve $C'$ intersects the ruling in a unique point with multiplicity $4$ that is also a singular point of $C'$, and is a $\mathbf{Z}_{12}$ singularity of $C$.
    \end{enumerate}
\end{proposition}
\begin{proof}
    By the discussion in Section \ref{sec: algorithm output}, Theorem \ref{thm: main for divisors} and Table \ref{tab:unstable-nonstable}, an unstable curve $C$ is given by the monomials of either $N^+(\lambda_3)$ or $N^+(\lambda_4)$, i.e. by one of the following two equations:
    \begin{align}   
        f_1 = \,\, &a_1XY^3Z^4 +a_2X^2Y^2Z^3W+ a_3X^2Y^2Z^4 +a_4X^3YZW^3 + a_5X^3YZ^2W^2 + \label{eq: unstable 1} \\
        &a_6X^3YZ^3W + a_7X^3YZ^4 + a_8X^4W^4 + a_9X^4ZW^3 + \notag \\
        &a_{10}X^4Z^2W^2 + a_{11}X^4Z^3W + a_{12}X^4Z^4  \notag \\
        = \, &X\big(a_1Y^3Z^4 +a_2XY^2Z^3W+ a_3XY^2Z^4 +a_4X^2YZW^3 + a_5X^2YZ^2W^2 + \notag \\
        &a_6X^2YZ^3W + a_7X^2YZ^4 + a_8X^3W^4 + a_9X^3ZW^3 + \notag \\
        &a_{10}X^3Z^2W^2 + a_{11}X^3Z^3W + a_{12}X^3Z^4 \big), \notag
        \shortintertext{(corresponding to $N^+(\lambda_4)$), or}
        f_2 = \,\, &a_1X^4Z^4 + a_2X^3YZ^2W^2 + a_3X^3YZ^3W + a_4X^4Z^3W + a_5X^2Y^2Z^4 + \label{eq: unstable 2} \\ 
        &a_6X^4Z^2W^2 + a_7X^3YZ^4 + a_8X^4ZW^3 + a_9X^3YW^4 + a_{10}X^2Y^2Z^3W + \notag \\ &a_{11}X^4W^4 + a_{12}X^3YZW^3 \notag \\
        = \, &X^2\big(a_1X^2Z^4 + a_2XYZ^2W^2 + a_3XYZ^3W + a_4X^2Z^3W + a_5Y^2Z^4 + \notag \\
        &a_6X^2Z^2W^2 + a_7XYZ^4 + a_8X^2ZW^3 + a_9X^2YW^2 + a_{10}Y^2Z^3W + \notag \\
        &a_{11}X^2W^4 + a_{12}XYZW^3 \big), \notag
    \end{align}  
    (corresponding to $N^+(\lambda_3)$), where the $a_i$ are coefficients which are not all zero.
    
    Equation \eqref{eq: unstable 1} represents a reducible curve $C$ with one ruling $l_1 = \{X=0\}$. The residual $(3,4)$-curve $C'$ is singular with $\mathbf{D}_5$ singularity at $P = ([0:1],[0:1])$; the intersection with the ruling is given by $C'\cap l_1 = P=([0:1], [0:1])$ with multiplicity $4$. The point $P$ is the unique singular point of $C'$, and the point $P$ is a $\mathbf{Z}_{12}$ singularity, i.e. the singularity with normal form $x^3y +xy^4 +ax^2y^3$    
    (see \cite[p. 93]{Arnold1976} for more details), as it is the intersection of a $\mathbf{D}_5$ singularity with a double line, as required. The normal form of this singularity can also be verified via the computational mathematical software MAGMA.

    Equation \eqref{eq: unstable 2} represents a reducible curve $C$ with a double ruling $\{X^2=0\}$. The residual $(2,4)$-curve $C'$ is smooth. The intersection $X\cap C'$ with the double ruling is the two points $([0:1],[0:1])$ and $([0:1],[1:-1])$. 
 \end{proof}

We will now classify strictly semistable $(4,4)$-curves.

 \begin{proposition}[Strictly semistable curves]\label{strictly semistable characterisation}
    Let $C$ be a $(4,4)$-curve in $\mathbb{P}^1\times\mathbb{P}^1$. Then $C$ is strictly semistable if and only if one of the following holds:
    \begin{enumerate}
        \item $C$ is irreducible with one $\mathbf{X}_9$ singularity (multiplicity $4$);
        \item $C$ is reducible, containing one ruling, with one $\mathbf{J}_{2,0}$ singularity (multiplicity $3$) at the intersection of the ruling and the residual curve;
        \item $C$ is reducible and non-reduced, containing a double ruling. The resulting $(2,4)$-curve $C'$ is smooth, intersecting the double ruling at 4 points given by solutions of the generic degree 4 form in $\pr^1$.
    \end{enumerate}
\end{proposition}
\begin{proof}
    From the above discussion, Theorem \ref{thm: main for divisors} and the Centroid Criterion (Theorem \ref{thm:centroid}), a curve $C$ is strictly semistable if and only if its equation is given by the monomials of $N^{\oplus}(\lambda_0)$ or $N^{\oplus}(\lambda_1)$ or $N^{\oplus}(\lambda_2)$. The monomials of $N^{\oplus}(\lambda_0)$ define a singular, reducible and non-reduced curve, with a double ruling $\{X^2=0\}$. The residual $(2,4)$-curve $C'$ is smooth and intersects the ruling $X\cap C'$ at $4$ points given by solutions of the generic degree 4 form in $\mathbb{P}^1$.
    
    The monomials of $N^{\oplus}(\lambda_1)$ define a singular irreducible curve with singularity at the point  $P=([0:1], [0:1])$, with multiplicity $4$. In particular, this singularity is an $\mathbf{X}_9$ singularity, via the mathematical software MAGMA.

    Similarly, the monomials of $N^{\oplus}(\lambda_2)$ define a singular reducible curve with singularity at the point  $P=([0:1], [0:1])$, with multiplicity $3$. This curve contains a $(1,0)$-ruling $l_1 = \{X = 0\}$. The residual curve $C'$ intersects $l_1$ at $P$, which is a $\mathbf{J}_{2,0}$ singularity since it has normal form $x^3+bx^2y^2+y^6$
    (see \cite[p. 95]{Arnold1976}), via MAGMA.
\end{proof}

We will now classify GIT strictly polystable orbits. By Proposition \ref{thm: polystable divisors}, if a curve $C$ belongs to a strictly polystable orbit, then there exist $g \in G, \lambda \in S_{1,1}^{4,4}$ such that $\operatorname{Supp}(g \cdot C) = \operatorname{Ann}(\lambda)$. From the discussion of Section \ref{sec: algorithm output}, we have three annihilator sets of monomials, $\operatorname{Ann}(\lambda_0)$, $\operatorname{Ann}(\lambda_1)$ and $\operatorname{Ann}(\lambda_2)$, that define curves with the following equations:

\begin{align}
    f_3 = \, \, &a_1X^2Y^2Z^4 + a_2X^2Y^2Z^3W + a_3X^2Y^2Z^2W^2 + a_4X^2Y^2ZW^3 + \label{eq: polystable 0} \\
    &a_5X^2Y^2W^4, \notag \\
    f_4 = \, \, &a_1X^2Y^2Z^2W^2 + a_2X^3YZW^3 + a_3X^4W^4 + a_4Y^4Z^4 + a_5XY^3Z^3W, \label{eq: polystable 1} \\
    \shortintertext{and}
    f_5 = \, \, &a_1X^2Y^2Z^2W^2 + a_2XY^3Z^4 + a_3X^3YW^4. \label{eq: polystable 2}
\end{align}

Equation \eqref{eq: polystable 0} defines a family of strictly semistable curves that are degenerations of the strictly semistable curves with double ruling (Proposition \ref{strictly semistable characterisation}.3). Equation \eqref{eq: polystable 1} defines a family of strictly semistable curves that are degenerations of the strictly semistable curves with singular point with multiplicity $4$ (Proposition \ref{strictly semistable characterisation}.1), while Equation \eqref{eq: polystable 2} defines a family of strictly semistable curves that are degenerations of the strictly semistable curves with singular point with multiplicity $3$ (Proposition \ref{strictly semistable characterisation}.2). 

We will first study the curves $C_3 = \{f_3 = 0\}$. Notice that 
$$f_3 = X^2Y^2(a_1Z^4 + a_2Z^3W + a_3Z^2W^2 + a_4ZW^3 + a_5W^4),$$
so that $C_3$ is a two-dimensional family of curves such that it is reducible and non-reduced and has two  double rulings $\{X^2=0\}$ and $\{Y^2=0\}$. The resulting $(0,4)$-curve $C'$ is smooth, and is the generic degree 4 polynomial in $\mathbb{P}^1$. The intersections $\{X=0\}\cap C'$ and $\{Y=0\}\cap C'$ are each given by $4$ points which are solutions of the generic degree $4$ polynomial in $\mathbb{P}^1$. We will call such curves \emph{maximally degenerate 4-curves}. Verifying that $C_3$ is polystable is simple; since $C_3$ is strictly semistable, the stabiliser subgroup of $C_3$, namely $G_{C_3}\subset G$, is infinite (see \cite[Remark8.1 (5)]{dolgachev_2003}). Hence, the dimension of the stabiliser is maximal, and the orbit is closed, i.e. $C_3$ is polystable. 

We next study the curves $C_5 = \{f_5 = 0\}$. Notice that up to $G$-action we can write 
$$f_5 = XY(XYZ^2W^2+Y^2Z^4+X^2W^4),$$
such that $C_5$ is the unique curve with two $(1,0)$-rulings $L_1 = \{X = 0\}$, $L_2 = \{Y = 0\}$, such that $C_4 = L_1\cup L_2 \cup C'$, where $C'$ is the residual singular $(2,4)$-curve with two $\mathbf{A}_3$ singularities at $P = ([0 : 1] , [0 : 1])$ and  $Q = ([1 : 0], [1 : 0])$, which decomposes into two cubics (see Figure \ref{fig:J-curves}). Furthermore, $L_1\cap C ' = P$ and $L_2\cap C ' = Q$ are the only singular points on $C_5$ which are $\mathbf{J}_{2,0}$ singularities (see \cite[p. 95]{Arnold1976} for more details). We will call this curve $C_5$ a \emph{maximally degenerate $\mathbf{J}_{2,0}$-curve}. To prove that $C_5$ is polystable, notice that the stabiliser subgroup of $C_5$, namely $G_{C_5}\subset G$, is infinite (see \cite[Remark 8.1 (5)]{dolgachev_2003}), i.e. the dimension of the stabiliser is maximal, and the orbit is closed. This implies that $C_5$ is polystable. 

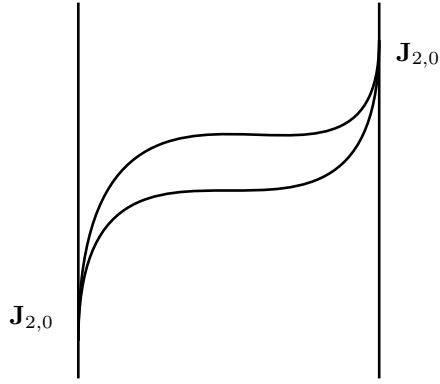
\begin{figure} 
    \centering
\begin{tikzpicture}
\draw[line width=1pt] (0,0) .. controls (0,4) and (4,0) .. (4,4);
\draw[line width=1pt] (0,0) .. controls (0,5) and (4,1) .. (4,4);
\draw[line width=1pt] (0,-0.5) -- (0,4.5);
\draw[line width=1pt] (4,-0.5) -- (4,4.5);
\node at (-0.2,0.3) [left] {$\mathbf{J}_{2,0}$};
  \node at (4.1,3.8) [right] {$\mathbf{J}_{2,0}$};
\end{tikzpicture}
\caption{A maximally degenerate $\mathbf{J}_{2,0}$-curve.}
\label{fig:J-curves}
\end{figure}

For the curves $C_4 = \{f_4 = 0\}$ defined by Equation \eqref{eq: polystable 1}, we have the following analysis. Observe that 
$$f_4 = L_1L_2L_3L_4
    = a_1X^2Y^2Z^2W^2 + a_2X^3YZW^3 + a_3X^4W^4 + a_4Y^4Z^4 + a_5XY^3Z^3W,$$
where the $L_i$ are homogeneous forms of bidegree $(1,1)$. This defines a $3$-dimensional family of strictly semistable points. We call such curves $X$-curves, because the generic $X$-curve, i.e., the curve where $L_i\neq L_j$ for all $i$, $j$, has two $\mathbf{X}_9$ singularities (locally the intersection of $4$ lines). Depending on the choice of coefficients, an $X$-curve can either be a quadruple conic (i.e. $L_i = L_j$ for all $i$, $j$), $4L$, a triple conic and a $(1,1)$-curve which are non-tangent,  $3L+L_1$, two non-tangent double conics, $2L_1+2L_2$, or a double conic and two $(1,1)$-curves which are non-tangent,  $2L+L_1+L_2$. In addition, $f_4$ is a degeneration of the strictly semistable curve with a singularity with multiplicity $4$ (Proposition \ref{strictly semistable characterisation}.2). For polystability, arguing as before, $G_{C_4}\subset G$ is infinite (see \cite[Remark 8.1 (5)]{dolgachev_2003}). Hence, the dimension of the stabiliser is maximal, and the orbit is closed, i.e. $C_4$ is polystable, regardless of coefficients. 

In summary, we obtain the following:

\begin{proposition}\label{polystable curves}
The orbit closure of every strictly semistable curve with a multiplicity $4$ singularity contains an $X$-curve described by Equation \eqref{eq: polystable 1}, i.e. either a union of four conics at two $\mathbf{X}_9$ singularities, or a triple conic and a $(1,1)$-curve, or a double conic and two $(1,1)$-curves or two double conics. In particular, these curves are GIT polystable.

The orbit closure of every strictly semistable curve with a multiplicity $3$ singularity contains a maximally degenerate $\mathbf{J}_{2,0}$-curve described by Equation \eqref{eq: polystable 2}. 

The orbit closure of every strictly semistable curve with a double ruling contains a maximally degenerate 4-curve described by Equation \eqref{eq: polystable 0}.
\end{proposition}

    

Having classified the unstable, semistable and polystable orbits, we are left to classify the stable orbits. To do so, we must first prove the following two lemmas.

\begin{lemma}\label{red with J_20}
    A reducible $(4,4)$-curve $C = l_1\cup C'$, where $l_1$ is a ruling, has a unique $\mathbf{J}_{2,0}$ singularity at the intersection $l_1\cap C'$ if and only if $C$ is defined by the monomials of $N^{\oplus}(\lambda_2)$.
\end{lemma}
\begin{proof}
    Firstly, notice that if $C$ is defined by the monomials of $N^{\oplus}(\lambda_2)$, then we have $C = l_1\cup C'$, where $l_1 = \{X = 0\}$, and $P = ([0:1], [0:1]) = C'\cap l_1$ is a $\mathbf{J}_{2,0}$ singularity.

    Conversely, we may assume without loss of generality that $l_1 = \{X = 0\}$, and $P = ([0:1], [0:1]) = C'\cap l_1$ is the $\mathbf{J}_{2,0}$ singularity. Then, since $P$ is a $\mathbf{J}_{2,0}$ singularity tangent at $l_1$, $C'$ must have an $\mathbf{A}_3$ singularity at $P$. This is only true, up to $G$-action, if $C$ is defined by the monomials of $N^{\oplus}(\lambda_2)$.
\end{proof}

\begin{lemma}\label{irr with X9}
    An irreducible $(4,4)$-curve $C$ has a unique $\mathbf{X}_9$ singularity if and only if $C$ is defined by the monomials of $N^{\oplus}(\lambda_1)$.
\end{lemma}
\begin{proof}
    For the first direction, notice that if $C$ is defined by the monomials of $N^{\oplus}(\lambda_1)$, then $P = ([0:1], [0:1]) = C'\cap l_1$ is a $\mathbf{X}_9$ singularity.

    For the converse, we may assume without loss of generality, that the $\mathbf{X}_{9}$ singularity is $P = ([0:1], [0:1])$. As $P\in C$, we have $C = \{f =0\}$, where $f$ does not contain the monomial $Z^4W^4$. Since $\operatorname{mult}(P) = 4$, $f$ cannot contain monomials $X^iY^{4-i}Z^jW^{4-j}$ such that $i=3$ and $j=0$, $i=2$ and $j<2$ or $i=1$ and  $j<3$, or $j=3$ and $i=0$, $j=2$ and $i<2$ or $j=1$ and  $i<3$. The remaining monomials are precisely the monomials in the set $N^{\oplus}(\lambda_1)$, as required.
\end{proof}

We are now in a position to classify all stable $(4,4)$-curves.

\begin{theorem}\label{stable 4,4 curves}
    Let $C$ be a $(4,4)$-curve in $\mathbb{P}^1\times\mathbb{P}^1$. Then $C$ is stable if and only if $C$ does not contain a double ruling and one of the following holds:
    \begin{enumerate}
        \item $C$ is an irreducible, possibly singular curve with singularities better than $\mathbf{X}_9$;
        \item $C$ is reducible, containing a $(1,0)$-ruling, and the residual $(3,4)$-curve $C'$ intersects the ruling in a unique point with multiplicity $\leq 3$ that is also a singular point of $C$, with singularities better than $\mathbf{J}_{2,0}$.
    \end{enumerate}
\end{theorem}
\begin{proof}
     Let $C$ be a stable $(4,4)$-curve in $\mathbb{P}^1\times\mathbb{P}^1$. From Propositions \ref{unstable characterisation} and \ref{strictly semistable characterisation} and Lemma \ref{irr with X9}, we see that if $C$ is irreducible, it cannot have an $\mathbf{X}_9$ or worse singularity. Similarly, from Propositions \ref{unstable characterisation} and \ref{strictly semistable characterisation} and Lemma \ref{red with J_20}, we see that if $C$ is reducible, then it cannot have more than one ruling, or a double ruling. If it has one ruling, then $C$ cannot be such that the residual $(3,4)$-curve $C'$ intersects the ruling in a unique point with multiplicity $3$ that is also a singular point of $C$, with $\mathbf{J}_{2,0}$ singularity or in a unique point with multiplicity $4$ that is also a singular point of $C$, with $\mathbf{Z}_{12}$ singularity.
\end{proof}

\subsection{A Connection to K-Moduli Wall-Crossing}\label{k-moduli}

We will conclude this section by tying the GIT classification we have obtained to known results regarding K-moduli wall-crossing. Let $X = \pr^1 \times \pr^1$ and consider the log Fano pair $\big(X, cC_{4,4}\big)$, where $C_{4,4}\in |-2K_X|$. Let $\mathcal{M}^K(c)$ be the K-moduli stack parametrising K-semistable log Fano pairs $\big(X, cC_{4,4}\big)$, with good moduli space $M^K(c)$, parametrising K-polystable log Fano pairs $\big(X, cC_{4,4}\big)$. Let also $\mathcal{M}^{GIT}$ be the GIT moduli stack parametrising GIT semistable $(4,4)$-curves in $X$, with good moduli space $M^{GIT}$ the GIT quotient $\pr|\mathcal{O}_{\pr^1 \times \pr^1}(4,4)|^*\sslash \big(\operatorname{SL}(2)\times \operatorname{SL}(2) \big)$ detailed in Section \ref{sec: algorithm output}.

By \cite[Theorem 1.1]{ascher2020kmoduli}, we know that for $0<c<1/8$ there exists an isomorphism of stacks $\mathcal{M}^K(c) \cong \mathcal{M}^{GIT}$, which descends to an isomorphism of good moduli spaces ${M}^K(c) \cong {M}^{GIT}$. Propositions \ref{strictly semistable characterisation} and  \ref{polystable curves} and Theorem \ref{stable 4,4 curves} allow us to obtain an explicit description of $\mathcal{M}^K(c)$ and ${M}^K(c)$.

\begin{theorem}\label{thm: wall-crossing}
    Let $0<c<1/8$. A log pair $\big(\mathbb{P}^1\times \mathbb{P}^1, cC_{4,4}\big)$ is 
    \begin{itemize}
        \item strictly K-semistable if and only if
        \begin{enumerate}
        \item $C_{4,4}$ is irreducible with one $\mathbf{X}_9$ singularity (multiplicity $4$);
        \item $C_{4,4}$ is reducible, containing one ruling, with one $\mathbf{J}_{2,0}$ singularity (multiplicity $3$) at the intersection of the ruling and the residual curve;
        \item $C_{4,4}$ is reducible and non-reduced, containing a double ruling. The resulting $(2,4)$-curve $C'$ is smooth, intersecting the double ruling at 4 points given by solutions of the generic degree 4 form in $\pr^1$,
    \end{enumerate}
    \item strictly K-polystable if and only if
    \begin{enumerate}
        \item $C_{4,4}$ is a reducible maximally degenerate $4$-curve described by Equation \eqref{eq: polystable 0};
        \item $C_{4,4}$ is an irreducible $X$-curve described by Equation \eqref{eq: polystable 1};
        \item $C_{4,4}$ is a reducible maximally degenerate $\mathbf{J}_{2,0}$-curve described by Equation \eqref{eq: polystable 2},
    \end{enumerate}
    \item strictly K-stable if and only if
    \begin{enumerate}
        \item $C_{4,4}$ is an irreducible, possibly singular curve with singularities better than $\mathbf{X}_9$;
        \item $C_{4,4}$ is reducible, containing a $(1,0)$-ruling , with singularities better than $\mathbf{J}_{2,0}$.
    \end{enumerate}
    \end{itemize}
\end{theorem}
\begin{proof}
    This is an immediate consequence of Propositions \ref{strictly semistable characterisation} and  \ref{polystable curves}, Theorem \ref{stable 4,4 curves} and \cite[Theorem 1.1]{ascher2020kmoduli}.
\end{proof}

\begin{remark}
    Combining the results of \cite[\S 4]{Shah_1981}, \cite[Lemma 3.2, Proposition 3.3]{Laza_OGrady_2018} and \cite[Theorem 1.1]{ascher2020kmoduli} one can recover the above classification of K-(semi/poly)stable points of Theorem \ref{thm: wall-crossing}, Theorem \ref{thm: wall-crossing} provides an explicit description of the K-(semi/poly)stable log Fano pairs in terms of explicit forms of singularities relying on the results of Section \ref{sec: stability analysis}, which is not present in the aforementioned works.
\end{remark}

\section{GIT of (1,2)-Divisors in \texorpdfstring{$\pr^1\times \pr^3$}{}}\label{1,2 divisors}

We will apply the GIT algorithm detailed above (Section \ref{sec: computer implementation}) to classify the GIT quotient of $(1,2)$-divisors in $\pr^1_{\mathbf{x}}\times \pr^3_{\mathbf{y}}$. Smooth $(1,2)$-divisors are general elements of Fano threefold family 2.25, and are known to be K-stable (c.f \cite[Corollary 4.32]{araujo_2023}). Let $G \coloneqq \operatorname{SL}(2) \times \operatorname{SL}(4)$, and $V \coloneqq |\mathcal{O}_{\pr^1 \times \pr^3}(1,2)|$. We will study the GIT quotient $\mathbb{P} V^* \sslash G$ computationally. 

\subsection{Algorithm Output}\label{sec: algorithm output 1,2}

Recall that we denote a normalised one-parameter subgroup $\lambda$ in $G$ by $\lambda = (u,-u,s_0,s_1, s_2,s_3)$. By our algorithm, we obtain $S_{1,1}^{1,2}$ as in Definition \ref{defn: new fundamental set} and Theorem \ref{thm: fundamental set}. This set has 429 elements, which we won't list here. The relevant one-parameter subgroups that give maximal semi-destabilised sets are 
\begin{align*}
\lambda_0 &= (0, 0, 1, 1, -1, -1),\\
\lambda_1 &= (2, -2, 1, 1, -1, -1),\\
\lambda_2 &= (2, -2, 1, 1, 1, -3),\\
\lambda_3 &= (0, 0, 1, 0, 0, -1),\\
\lambda_4 &= (2, -2, 3, -1, -1, -1).
\end{align*}

These give five separate maximal semi-destabilising sets $N^{\oplus}(\lambda_i)$, which we will analyse in order to determine the stable, polystable and semistable orbits. These are given by the sets of monomials detailed in Table \ref{tab:unonstable 1,2}.

\begingroup
\renewcommand{\arraystretch}{1.05}
\begin{table}[h]
   \centering
   \begin{NiceTabular}{|c|l|l|l|l|l|}
     \hline
      Families & $N^{\oplus}(\lambda_0)$ & $N^{\oplus}(\lambda_1)$  & $N^{\oplus}(\lambda_2)$& $N^{\oplus}(\lambda_3)$ & $N^{\oplus}(\lambda_4)$  \\ \hline
      \Block{18-1}{\rotate monomials}
       & $x_0y_0^2$ & $x_0y_0^2$ & $x_0y_0^2$ & $x_0y_0^2$ & $x_0y_0^2$ \\
       & $x_0y_0y_1$ & $x_0y_0y_1$ & $x_0y_0y_1$ & $x_0y_0y_1$ & $x_0y_0y_1$ \\
       & $x_0y_0y_2$ & $x_0y_0y_2$ & $x_0y_0y_2$ & $x_0y_0y_2$ & $x_0y_0y_2$ \\
       & $x_0y_0y_3$ & $x_0y_0y_3$ & $x_0y_0y_3$ & $x_0y_0y_3$ & $x_0y_0y_3$ \\
       & $x_0y_1^2$ & $x_0y_1^2$ & $x_0y_1^2$ & $x_0y_1^2$ & $x_0y_1^2$ \\
       & $x_0y_1y_2$ & $x_0y_1y_2$ & $x_0y_1y_2$ & $x_0y_1y_2$ & $x_0y_1y_2$ \\
       & $x_0y_1y_3$ & $x_0y_1y_3$ & $x_0y_1y_3$ & --- & $x_0y_1y_3$ \\
       & --- & $x_0y_2^2$ & $x_0y_2^2$ & $x_0y_2^2$ & $x_0y_2^2$ \\
       & --- & $x_0y_2y_3$ & $x_0y_2y_3$ & --- & $x_0y_2y_3$ \\
       & --- & $x_0y_3^2$ & --- & --- & $x_0y_3^2$ \\
       & $x_1y_0^2$ & $x_1y_0^2$ & $x_1y_0^2$ & $x_1y_0^2$ & $x_1y_0^2$ \\
       & $x_1y_0y_1$ & $x_1y_0y_1$ & $x_1y_0y_1$ & $x_1y_0y_1$ & $x_1y_0y_1$ \\
       & $x_1y_0y_2$ & --- & $x_1y_0y_2$ & $x_1y_0y_2$ & $x_1y_0y_2$ \\
       & $x_1y_0y_3$ & --- & --- & $x_1y_0y_3$ & $x_1y_0y_3$ \\
       & $x_1y_1^2$ & $x_1y_1^2$ & $x_1y_1^2$ & $x_1y_1^2$ & --- \\
       & $x_1y_1y_2$ & --- & $x_1y_1y_2$ & $x_1y_1y_2$ & --- \\
       & $x_1y_1y_3$ & --- & --- & --- & --- \\
       & --- & --- & $x_1y_2^2$ & $x_1y_2^2$ & --- \\
       \hline
   \end{NiceTabular}
   \caption{Not-stable families and their monomials.}
   \label{tab:unonstable 1,2}
\end{table} 
\endgroup
  
By Theorem \ref{thm:centroid}, all families $N^{\oplus}(\lambda_i)$ represent strictly semistable divisors. Their potential closed orbits, i.e. the monomials of zero weight given by $\operatorname{Ann}(\lambda_i)$, are listed in Table \ref{tab:closed orbits 1,2}.

\begingroup
\renewcommand{\arraystretch}{1.1}
\begin{table}[h]
   \centering
   \begin{NiceTabular}{|c|l|l|l|l|l|}
     \hline
      Families & $\operatorname{Ann}(\lambda_0)$ & $\operatorname{Ann}(\lambda_1)$& $\operatorname{Ann}(\lambda_2)$& $\operatorname{Ann}(\lambda_3)$& $\operatorname{Ann}(\lambda_4)$ \\ 
      \hline
      \Block{9-1}{\rotate monomials}
      & $x_0y_1y_3$&$x_0y_3^2$ & $x_0y_2y_3$ & $x_0y_2^2$& $x_0y_3^2$ \\
      & $x_0y_1y_2$& $x_0y_2y_3$& $x_0y_1y_3$& $x_0y_1y_2$ & $x_0y_2y_3$ \\
      & $x_0y_0y_3$& $x_0y_2^2$& $x_0y_0y_3$& $x_0y_1^2$  & $x_0y_2^2$\\
      & $x_0y_0y_2$& $x_1y_1^2$& $x_1y_2^2$& $x_0y_0y_3$ & $x_0y_1y_3$\\
      & $x_1y_1y_3$&$x_1y_0y_1$ & $x_1y_1y_2$& $x_1y_2^2$ & $x_0y_1y_2$\\
      & $x_1y_1y_2$& $x_1y_0^2$& $x_1y_1^2$& $x_1y_1y_2$& $x_0y_1^2$\\
      & $x_1y_0y_3$&          & $x_1y_0y_2$& $x_1y_1^2$ & $x_1y_0y_3$\\
      & $x_1y_0y_2$&          & $x_1y_0y_1$& $x_1y_0y_3$ & $x_1y_0y_2$\\
      &            &          & $x_1y_0^2$ &             & $x_1y_0y_1$\\
      \hline
    \end{NiceTabular}
    \caption{Potential closed orbits.}
    \label{tab:closed orbits 1,2}
\end{table}
\endgroup

Notice that, for a threefold $X = \{f = 0\}$ such that $\operatorname{Supp}(X) = N^{\oplus}(\lambda_3)$, $f$ is given by 
$$f = x_0\big(f_2(y_0,y_1,y_2) + a_0y_0y_3\big) + x_1\big(g_2(y_0,y_1,y_2) +a_1y_0y_3\big),$$
where the $f_2$ and $g_2$ are degree 2 homogeneous polynomials in $y_0$, $y_1$, $y_2$ and the $a_i$ are non-zero coefficients. In particular, up to a suitable $\Aut(\pr^1\times \pr^3)$-action, $f$ is given by 
 \begin{align}\label{n0l1}
        f &= x_0\big(f_2(y_0,y_1,y_2) + y_0y_3\big) + x_1\big(g_2(y_0,y_1,y_2) - y_0y_3\big).
    \end{align}

\subsection{Stability Analysis}\label{sec: stability analysis 1,2}

We will first prove some general results regarding singular $(1,2)$-divisors. A general $(1,2)$-divisor $X$ in $\pr^1\times \pr^3$ is given by an equation
\begin{equation}\label{eq: generic form of 1,2 divisor}
    f = x_0f_2(y_0,y_1,y_2,y_3) +x_1g_2(y_0,y_1,y_2,y_3),
\end{equation}
where $f_2$ and $g_2$ are degree $2$ homogeneous polynomials in variables $y_0$, $y_1$, $y_2$, $y_3$. Consider the complete intersection $C = f_2\cap g_2$ in $\pr^3$. We will now relate the singular points of $X$ to the singular points of $C$.

\begin{lemma}\label{1,2 divisors and pencils}
    Let $X = \{f=0\}$ be a $(1,2)$-divisor in $\pr^1\times \pr^3$. A point $P = ([1:1], [a:b:c:d])\in X$ is singular if and only if $[a:b:c:d]\in C = f_2\cap g_2$ is a singular point of the complete intersection.
\end{lemma}
\begin{proof}
    Firstly, let $P = ([1:1], [a:b:c:d])\in X$ be a singular point, where $X$ is given by the equation 
    $$f = x_0f_2(y_0,y_1,y_2,y_3) +x_1g_2(y_0,y_1,y_2,y_3).$$
    Since $P$ is singular, all the partial derivatives $\frac{\partial f}{\partial x_i}$, $\frac{\partial f}{\partial y_i}$ must vanish at $P$. Hence, 
    $$f_2(a:b:c:d) =0=g_2(a:b:c:d),$$
    and $P'\coloneqq[a:b:c:d]\in C\coloneqq f_2\cap g_2$. In addition, we get
    $$\frac{\partial f_2}{\partial y_i}\Bigg\vert_{P'} = -\frac{\partial g_2}{\partial y_i}\Bigg\vert_{P'}.$$
    But then, the matrix
    $$J = \begin{pmatrix}
        \frac{\partial f_2}{\partial x_i}\Big\vert_{P'}\\
        \frac{\partial g_2}{\partial x_i}\Big\vert_{P'}
        \end{pmatrix}$$
at $[a:b:c:d]$ has $\mathrm{rank}(J)=1<2$, hence $[a:b:c:d]\in \operatorname{Sing}(C)$, as required.

Conversely, let $P' \coloneqq [a:b:c:d]\in \operatorname{Sing}(C)\subset C\coloneqq f_2\cap g_2$. This implies that the matrix
    $$J = \begin{pmatrix}
        \frac{\partial f_2 }{\partial x_i}\Big\vert_{P'}\\
        \frac{\partial g_2}{\partial x_i}\Big\vert_{P'}
        \end{pmatrix}$$
at $[a:b:c:d]$ has $\mathrm{rank}(J)<2$, i.e. there exists a constant $c$, such that 
\begin{equation}\label{eq: partials are proportional}
    \frac{\partial f_2}{\partial y_i}\Bigg\vert_{P'} = c\frac{\partial g_2}{\partial y_i}\Bigg\vert_{P'}.
\end{equation}
Let $X$ be given by the equation 
$$f = x_0f_2(y_0,y_1,y_2,y_3) +x_1g_2(y_0,y_1,y_2,y_3).$$
Since $[a:b:c:d]\in C$, we have that $P\coloneqq([1:1], [a:b:c:d])\in X$, and that 
$$\frac{\partial f }{\partial x_0}\Bigg\vert_P= \frac{\partial f }{\partial x_1}\Bigg\vert_P =0.$$
Furthermore, by applying Equation \eqref{eq: partials are proportional} and scaling accordingly, we have 
$$\frac{\partial f }{\partial y_i}\Bigg\vert_P = x_0\frac{\partial f_2 }{\partial y_i}\Bigg\vert_P+ x_1\frac{\partial g_2 }{\partial y_i}\Bigg\vert_P = 0,$$
hence $P\in \operatorname{Sing}(X)$.
\end{proof}

A direct consequence is the following.

\begin{corollary}\label{type of singularity remains same}
    Let $X = \{f=0\}$ be a singular $(1,2)$-divisor in $\pr^1\times \pr^3$ given by Equation \eqref{eq: generic form of 1,2 divisor}. Then the singular point $P = ([1:1], [a:b:c:d])\in X$, is a singular point of type $\mathbf{A}_n$ or $\mathbf{D}_n$ if and only if the singular point $[a:b:c:d]\in f_2\cap g_2$ is of type $\mathbf{A}_n$ or $\mathbf{D}_n$.
\end{corollary}

The above two results, combined with the results of \cite[\S 4.2, Table 2]{pap22} allow us to obtain a full description of singularities of $(1,2)$-divisors in $\pr^1\times \pr^3$ in terms of equations (up to $\Aut(\pr^1\times \pr^3)$-equivalence).

\begin{lemma}\label{lem: 1,2 divisors with A1 sings}
    Let $X = \{f=0\}$ be a singular $(1,2)$-divisor in $\pr^1\times \pr^3$ given by Equation \eqref{eq: generic form of 1,2 divisor}. Then $X$ has
    \begin{enumerate}
        \item one $\mathbf{A}_1$ singularity if and only if (up to $\Aut(\pr^1\times \pr^3)$-action)
        \begin{align}
             f_2&= q_1(y_0,y_1,y_2)+y_0y_3, \label{eq: 1,2 ss 1}\\
             g_2&= q_2(y_0,y_1,y_2)+y_0y_3; \notag 
             \shortintertext{\item two $\mathbf{A}_1$ singularities if and only if (up to $\Aut(\pr^1\times \pr^3)$-action)}
             f_2&= q_1(y_0,y_1,y_2,y_3), \label{eq: 1,2 ss 2}\\
             g_2&= q_2(y_0,y_1), \notag
             \shortintertext{or}
             f_2&= q_1(y_0,y_1)+y_0y_3,\label{eq: 1,2 ss 3}\\
             g_2&= y_0l_3(y_0,y_1,y_2,y_3), \notag
             \shortintertext{or}
             f_2&= q_1(y_0,y_1)+y_2l_1(y_0,y_1)+ y_3l_2(y_0,y_1), \label{eq: 1,2 ss 4}\\
             g_2&= q_2(y_0,y_1)+y_2l_3(y_0,y_1)+ y_3l_4(y_0,y_1); \notag
             \shortintertext{\item four $\mathbf{A}_1$ if and only if (up to $\Aut(\pr^1\times \pr^3)$-action)}
             f_2&= q_1(y_0,y_1),\label{eq: 1,2 ss 5}\\
             g_2&= q_2(y_2,y_3). \notag
        \end{align}
    \end{enumerate}
\end{lemma}
\begin{proof}
    Follows directly from Lemma \ref{1,2 divisors and pencils}, Corollary \ref{type of singularity remains same} and \cite[\S 4.2, Table 2]{pap22}.
\end{proof}

\begin{remark}
    We can proceed in the same way to classify all $(1,2)$-divisors with $\mathbf{A}_2$, $\mathbf{A}_3$, and $\mathbf{D}_4$ singularities, but we omit this classification, as it is not needed in our analysis.
\end{remark}

We may also obtain an explicit description of singular $(1,2)$-divisors.

\begin{lemma}\label{singular 1,2 divisors}
    Let $X = \{f=0\}$ be a $(1,2)$-divisor in $\pr^1\times \pr^3$. Then $X$ is singular if and only if, up to an $\Aut(\pr^1\times \pr^3)$-action, $f$ is given by 
    \begin{align}\label{most general singular 1,2}
        f &= x_0\big(f_2(y_0,y_1,y_2) + y_0y_3\big) + x_1\big(g_2(y_0,y_1,y_2) - y_0y_3\big),
    \end{align}
    or a degeneration of $f$, where the $f_2$ and $g_2$ are degree 2 forms.
\end{lemma}
\begin{proof}
    Let $X = \{f=0\}$ be a singular $(1,2)$-divisor in $\pr^1\times \pr^3$. Without loss of generality, we may assume that the singular point on $X$ is $P = ([1:1], [0:0:0:1])$. Since $P\in X$, $f$ cannot contain monomials of the form $x_0y_3^2$ or $x_1y_3^2$, i.e. $f$ is given by the equation 
    $$f = x_0\big(f_2(y_0,y_1,y_2) + y_3(a_0y_0+a_1y_1+a_2y_2)\big)+x_1\big(g_2(y_0,y_1,y_2) + y_3(b_0y_0+b_1y_1+b_2y_2)\big),$$
    where the $f_2$ and $g_2$ are degree 2 forms, and the $a_i$ and $b_i$ are non-zero coefficients. 
    
    We now compute the six partial derivatives of $f$ with respect to the coordinates $x_i$, $y_i$.
    Since $P\in \operatorname{Sing}(X)$, these six equations must all vanish at $P$. Substituting for $P$, we get $a_i = -b_i$ for all $i$. Performing the change of variables $y_0 \rightarrow a_0y_0 + a_1y_1 + a_2y_2$, which leaves $P$ invariant, we obtain
    $$f = x_0(f_2(y_0,y_1,y_2) + y_0y_3) + x_1(g_2(y_0,y_1,y_2) - y_0y_3),$$
    where, by abuse of notation, $f_2$ and $g_2$ are the transformed forms of degree $2$. But this is precisely Equation \eqref{most general singular 1,2}. Since we started from a general point and equation, the proof is complete.
\end{proof}

We will now classify GIT strictly semistable divisors. 

\begin{proposition}
\label{semistable 1,2 divisors}
    Let $X$ be a $(1,2)$-divisor in $\pr^1\times \pr^3$. Then $X$ is GIT strictly semistable if and only if $X$ is either
    \begin{enumerate}
        \item  an irreducible singular Fano threefold with one $\mathbf{A}_1$ singularity, such that $\operatorname{Supp}(X) = N^{\oplus}(\lambda_3)$, or
        \item an irreducible singular Fano threefold, with two $\mathbf{A}_1$ singularities, such that $\operatorname{Supp}(X) = N^{\oplus}(\lambda_0)$, or $\operatorname{Supp}(X) = N^{\oplus}(\lambda_1)$, or $\operatorname{Supp}(X) = N^{\oplus}(\lambda_2)$, or $\operatorname{Supp}(X) = N^{\oplus}(\lambda_4)$.
    \end{enumerate}
\end{proposition}
\begin{proof}
    By Theorem \ref{thm: main for divisors} and the Centroid Criterion (Theorem \ref{thm:centroid}), we know that $X$ is strictly semistable if and only if $\operatorname{Supp}(X) = N^{\oplus}(\lambda_i)$, for $0\leq i\leq 4$.
    Notice that a divisor $X$ that has $\operatorname{Supp}(X) = N^{\oplus}(\lambda_i)$ for all $i$, corresponds to each of the four possibilities presented in Lemma \ref{lem: 1,2 divisors with A1 sings}.1 and \ref{lem: 1,2 divisors with A1 sings}.2. Hence, $X$ has either one or two $\mathbf{A}_1$ singularities. If $X$ has two $\mathbf{A}_1$ singularities, then the monomials in its support correspond to one of $N^{\oplus}(\lambda_0)$, $N^{\oplus} (\lambda_1)$, or $N^{\oplus}(\lambda_4)$ (and $X$ is given by Equation \eqref{eq: 1,2 ss 4}, \eqref{eq: 1,2 ss 2}, or \eqref{eq: 1,2 ss 3} respectively). If $X$ has one $\mathbf{A}_1$ singularity, then the monomials in its support are given by $N^{\oplus}(\lambda_3)$ or $N^{\oplus}(\lambda_2)$, thus $X$ is given by Equation \eqref{eq: 1,2 ss 1}.
\end{proof}

\begin{corollary}
    A $(1,2)$-divisor in $\pr^1\times \pr^3$ is strictly semistable if and only if it has $\mathbf{A}_1$ singularities.
\end{corollary}

An immediate consequence of Proposition \ref{semistable 1,2 divisors} and Lemma \ref{singular 1,2 divisors} is the characterisation of all GIT stable points. 

\begin{theorem}\label{stable 1,2 divisors}
    Let $X$ be a $(1,2)$-divisor in $\pr^1\times \pr^3$. Then $X$ is GIT stable if and only if $X$ is smooth.
\end{theorem}
\begin{proof}
    We will show that all singular $(1,2)$-divisors are not stable. Let $X$ be an arbitrary singular $(1,2)$-divisor in $\pr^1\times \pr^3$. Then, by Lemma \ref{singular 1,2 divisors}, $X = \{f = 0\}$, where $f$ is given by Equation \eqref{most general singular 1,2}. But notice that, by the discussion in Section \ref{sec: algorithm output 1,2}, we have $\operatorname{Supp}(X) = N^{\oplus}(\lambda_3)$. Hence, by Theorem \ref{thm: main for divisors}, $X$ is not stable. Since $X$ was an arbitrary singular $(1,2)$-divisor, this implies that all singular $(1,2)$-divisors are not stable. Conversely, if $X$ is a smooth $(1,2)$-divisor, it is not not-stable, i.e., it is stable.
\end{proof}

Using Proposition \ref{thm: polystable divisors} we will also classify GIT strictly polystable divisors.

\begin{proposition}
    \label{polystable 1,2 divisors}
    Let $X$ be a $(1,2)$-divisor in $\pr^1\times \pr^3$. Then $X$ is GIT strictly polystable if and only if, up to an $\Aut(\pr^1 \times \pr^3)$-action, $X$ is given by 
    $$f= (y_1y_2 -y_0y_3)x_0 + (-y_1y_2 - y_0y_3)x_1,$$
    and $X$ is a toric Fano threefold with four $\mathbf{A}_1$ singularities.
\end{proposition}
\begin{proof}
    By Proposition \ref{thm: polystable divisors}, we know that $X$ is strictly polystable if and only if $\operatorname{Supp}(X) = \operatorname{Ann}(\lambda_i)$, for some $0\leq i\leq 4$. 
    Notice that, up to an $\Aut(\pr^1\times \pr^3)$-action, the condition $\operatorname{Supp}(X) = \operatorname{Ann}(\lambda_i)$ yields the same singular threefold. For instance, let $\operatorname{Supp}(X) = \operatorname{Ann}(\lambda_3)$. Then $X = \{f = 0\}$, where $f$ is given by the equation
    \begin{align}\label{eq: 1,2 polystable}
        f &=  (d_2y_1^2 + d_1y_1y_2 + d_0y_2^2 + d_3y_0y_3)x_0 + (d_6y_1^2 + d_5y_1y_2 + d_4y_2^2 +d_7y_0y_3)x_1, 
    \end{align}
    where the $d_i$ are non-zero coefficients. We can choose a suitable $\Aut(\pr^1\times \pr^3)$-action such that the equation is given by 
    \begin{equation}\label{eq: 1,2 polystable refined}
        f =  (y_1y_2 -y_0y_3)x_0 + (-y_1y_2 - y_0y_3)x_1.
    \end{equation}
    Then $X$ is the unique irreducible toric Fano threefold, with four $\mathbf{A}_1$ singularities at
    $P = ([1:-1],[1:0:0:0])$, $Q = ([1:1],[0:1:0:0])$, $R = ([1:1],[0:0:1:0])$ and $S = ([1:-1],[0:0:0:1])$.
    
    Now let $\operatorname{Supp}(X) = \operatorname{Ann}(\lambda_0)$. Then $X = \{f = 0\}$, where $f$ is given by the equation 
    \begin{align}\label{eq: 1,2 polystable second description}
        f = \, \, &(d_3y_0y_2 + d_1y_1y_2 + d_2y_0y_3 + d_0y_1y_3)x_0 + \\
        &(d_7y_0y_2 +d_5y_1y_2+ d_6y_0y_3 + d_4y_1y_3)x_1, \notag
    \end{align}
    where the $d_i$ are non-zero coefficients. As before, it is possible to pick a suitable $\Aut(\pr^1\times \pr^3)$-action such that the equation is given by \eqref{eq: 1,2 polystable refined}. We can proceed in the same way with the other sets  $\operatorname{Ann}(\lambda_i)$, for $i=1$, $2$, $3$, to show that there exists a unique strictly polystable threefold up to $\Aut(\pr^1\times \pr^3)$-action, which is given by Equation \eqref{eq: 1,2 polystable refined}. Furthermore, it is the unique singular threefold with four $\mathbf{A}_1$ singularities, as required.
\end{proof}

\subsection{An Alternative Compactification of the K-Moduli of Fano Threefolds 2.25}\label{k-moduli 2.25}
In \cite[\S 5]{pap22}, the third author obtained an explicit description of the K-moduli space of Fano threefolds 2.25, using a different GIT description. Smooth Fano threefolds in family 2.25 are also blow ups of complete intersections of two quadrics in $\pr^3$. By classifying the GIT quotient of complete intersections of two quadrics in $\pr^3$, the third author showed that the K-moduli stack $\mathcal{M}^K_{2.25}$  parametrising K-semistable members of family 2.25, with good moduli space ${M}^K_{2.25}$, is isomorphic to the GIT quotient stack $\mathcal{M}^{GIT}_{3,2,2}$ parametrising GIT semistable complete intersections of two quadrics in $\pr^3$, with GIT quotient ${M}^{GIT}_{3,2,2}$. In this Section, we will obtain an alternative description for $\mathcal{M}^K_{2.25}$ and ${M}^K_{2.25}$, using the GIT moduli stack $\mathcal{M}^{GIT}_{1,2}$ parametrising GIT semistable $(1,2)$-divisors in $\pr^1\times \pr^3$, and the GIT quotient ${M}^{GIT}_{1,2}$ we classified in Section \ref{sec: stability analysis 1,2}.

Notice that the unique GIT polystable threefold $\tilde{X}$ described in Proposition \ref{polystable 1,2 divisors} is the unique singular toric Fano threefold with four $\mathbf{A}_1$ singularities. This threefold has polyhedron with barycentre $0$ (c.f. \cite[\S 5]{pap22}). Hence, by \cite{batyrev}, {\cite[Theorem 1.2]{fujita_volume}}, {\cite[Corollary 1.2.]{Berman}}, the Futaki invariant on $X$ vanishes and $X$ admits a K{\"a}hler-Einstein metric. Therefore, $X$ is K-polystable. In particular, this is the terminal toric Fano threefold with Reflexive ID $\#199$ in the Graded Ring Database (GRDB) ($3$-fold $\#255743$) \cite{Kas08}, which was shown to be K-polystable in \cite[Lemma 5.1]{pap22}. Furthermore, Theorem \ref{stable 1,2 divisors} shows that smooth $(1,2)$-divisors in $\pr^1\times \pr^3$, i.e. smooth Fano threefolds in family 2.25 are GIT stable. These are known to be K-stable by \cite[Corollary 4.32]{araujo_2023}. We will use the reverse moduli continuity method  \cite[\S 5]{pap22} to prove the following results.

\begin{lemma}\label{strictly k-ss for 2.25}
Let $X$, be a GIT strictly semistable $(1,2)$-divisor in $\pr^1\times \pr^3$. Then $X$ is strictly K-semistable.
\end{lemma}
\begin{proof}
Let $X$ be a GIT strictly semistable $(1,2)$-divisor in $\pr^1\times \pr^3$. Since $X$ is GIT strictly semistable, there exists a one-parameter subgroup $\lambda$ such that we have $\lim_{t\to 0} \lambda(t)\cdot X = \tilde{X}$, where $\tilde{X}$ is the unique GIT strictly polystable toric Fano threefold with 
four $\mathbf{A}_1$ singularities, described in Proposition \ref{polystable 1,2 divisors}. This one-parameter subgroup induces a family $f\colon \mathcal{X} \to \pr^1$, over $\pr^1$, such that the fibers $\mathcal{X}_t$ are isomorphic to $X$ for all $t\neq 0$, and $\mathcal{X}_0\cong \tilde{X}$. We thus have a map $\mathcal{X} \to \pr^1$ which is naturally a test configuration of $\mathcal{X}$ with central fibre $\mathcal{X}_0$.

Hence, we have constructed a test configuration $f\colon \mathcal{X} \to B$, where the central fiber $\mathcal{X}_0\cong \tilde{X}$ is a klt Fano threefold, which is K-polystable by the discussion above. Furthermore, the general fiber $\mathcal{X}_t\cong X$ is not isomorphic to $\mathcal{X}_0$. By \cite[Theorem 1.1]{blum-xu} and \cite{blum2021openness}, the central fiber $\mathcal{X}_t \cong X$ is strictly K-semistable. 
\end{proof}

We are now in a position to describe the K-moduli space ${M}^K_{2.25}$ in terms of the GIT quotient ${M}^{GIT}_{1,2}$.

\begin{theorem}\label{thm: 2.25 compactification}
There exists an isomorphism $\mathcal{M}^K_{2.25} \cong \mathcal{M}^{GIT}_{1,2}$. In particular, there exists an isomorphism ${M}^K_{2.25} \cong {M}^{GIT}_{1,2}\cong \pr^1$.
\end{theorem}
\begin{proof}
    We argue as in \cite[Theorem 5.5]{pap22}. By Theorem 
\ref{stable 1,2 divisors}, Proposition \ref{polystable 1,2 divisors}, Lemma \ref{strictly k-ss for 2.25} and the discussion at the start of this Section, we have an open immersion of a representable morphism of stacks:
\begin{center}
    \begin{tikzcd}
     \mathcal{M}^{GIT}_{1,2}\arrow[r, "\phi"] & \mathcal{M}^K_{2.25}\\
    \left[X\right] \arrow[r, mapsto, "\phi"] & \left[X\right].
    \end{tikzcd}
\end{center}

Since $\phi$ is an open immersion of stacks, by \cite[Lemma 06MY]{stacks-project}, $\phi$ is separated. Moreover, since it is injective, it is also quasi-finite. We now need to check that $\phi$ is an isomorphism that descends (as an isomorphism of schemes) to the moduli spaces
\begin{center}
    \begin{tikzcd}
     {M}^{GIT}_{1,2}\arrow[r, "\overline{\phi}"] & M^K_{2.25}\\
    \left[X\right] \arrow[r, mapsto, "\overline{\phi}"] & \left[X\right].
    \end{tikzcd}
\end{center}
As we have an open immersion $\phi$ of stacks, both of which admit moduli spaces, $\overline \phi$ is injective. 
Now, by \cite[Prop 6.4]{Alper}, since $\phi$ is representable, quasi-finite and separated, $\overline \phi$ is finite and $\phi$ maps closed points to closed points, we obtain that $ \phi$ is finite. Thus, by Zariski's Main Theorem, as $\overline \phi$ is a birational morphism with finite fibers to a normal variety, $\phi$ is an isomorphism to an open subset. But it is also an open immersion, and it is therefore an isomorphism.
\end{proof}

\begin{remark}
    Although the K-moduli description for family 2.25 is already known, we expect that the above computational approach and the developed software \cite{our_code} can be applied in various examples of K-moduli of Fano threefold families that arise as divisors of bidegree $(k,l)$ in products of projective space.
\end{remark}

\printbibliography
\end{document}